\documentclass[a4paper,leqno,12pt]{article}

\usepackage[utf8]{inputenc}
\usepackage[T1]{fontenc}
\usepackage[english]{babel}

\usepackage{crimson}

\usepackage{geometry}

\usepackage{fancyhdr}

\usepackage{hyperref}
\hypersetup{
  colorlinks   = true, 
  urlcolor     = blue, 
  linkcolor    = blue, 
  citecolor   = black 
}

\usepackage[inline]{enumitem}

\usepackage{graphicx}

\usepackage{color}

\usepackage[intlimits]{amsmath}
\usepackage{amssymb, amsfonts}

\usepackage[thmmarks, amsmath, amsthm]{ntheorem}

\usepackage{MnSymbol}
\usepackage{mathbbol}

\usepackage{bbm}

\usepackage{mathrsfs}

\usepackage[all]{xy}

\numberwithin{equation}{section}
\newtheorem{theoremcounter}{theoremcounter}[section]

\theoremstyle{plain}

\newtheorem{corollary}[theoremcounter]{Corollary}
\newtheorem{lemma}[theoremcounter]{Lemma}
\newtheorem{proposition}[theoremcounter]{Proposition}

\newtheorem{theorem}[theoremcounter]{Theorem}
\newtheorem{claim}{Claim}

\theoremnumbering{Alph}

\theoremstyle{definition}

\newtheorem{definition}[theoremcounter]{Definition}

\theoremstyle{remark}

\newtheorem{example}[theoremcounter]{Example}

\newtheorem{problem}[theoremcounter]{Problem}

\newtheorem{remark}[theoremcounter]{Remark}

\usepackage{xargs}                      
\usepackage[pdftex,dvipsnames]{xcolor}  
\usepackage[colorinlistoftodos,prependcaption,textsize=tiny]{todonotes}
\newcommandx{\unsure}[2][1=]{\todo[linecolor=red,backgroundcolor=red!25,bordercolor=red,#1]{#2}}
\newcommandx{\change}[2][1=]{\todo[linecolor=blue,backgroundcolor=blue!25,bordercolor=blue,#1]{#2}}
\newcommandx{\info}[2][1=]{\todo[linecolor=OliveGreen,backgroundcolor=OliveGreen!25,bordercolor=OliveGreen,#1]{#2}}
\newcommandx{\improvement}[2][1=]{\todo[linecolor=Plum,backgroundcolor=Plum!25,bordercolor=Plum,#1]{#2}}

\usepackage[nameinlink, capitalize, noabbrev]{cleveref}

\usepackage[style=alphabetic]{biblatex}
\usepackage{csquotes}

\DeclareFieldFormat{eprint}{\href{https://arxiv.org/abs/#1}{\texttt{#1}}}

\renewbibmacro*{in:}{}

\DeclareFieldFormat
[article,inbook,incollection,inproceedings,patent,thesis,unpublished,misc]
{title}{#1}

\newcommand{\cC}{\ensuremath{\mathcal{C}}}

\newcommand{\cF}{\ensuremath{\mathcal{F}}}
\newcommand{\cG}{\ensuremath{\mathcal{G}}}
\newcommand{\cH}{\ensuremath{\mathcal{H}}}

\newcommand{\cK}{\ensuremath{\mathcal{K}}}

\newcommand{\cN}{\ensuremath{\mathcal{N}}}
\newcommand{\cO}{\ensuremath{\mathcal{O}}}
\newcommand{\cP}{\ensuremath{\mathcal{P}}}

\newcommand{\cS}{\ensuremath{\mathcal{S}}}

\newcommand{\cU}{\ensuremath{\mathcal{U}}}
\newcommand{\cV}{\ensuremath{\mathcal{V}}}

\newcommand{\bC}{\ensuremath{\mathbb{C}}}

\newcommand{\bN}{\ensuremath{\mathbb{N}}}

\newcommand{\bR}{\ensuremath{\mathbb{R}}}

\newcommand{\rC}{\ensuremath{\mathrm{C}}}

\newcommand{\rE}{\ensuremath{\mathrm{E}}}

\newcommand{\rH}{\ensuremath{\mathrm{H}}}

\newcommand{\rK}{\ensuremath{\mathrm{K}}}

\newcommand{\rM}{\ensuremath{\mathrm{M}}}

\newcommand{\rmd}{\ensuremath{\mathrm{d}}}

\newcommand{\rmr}{\ensuremath{\mathrm{r}}}
\newcommand{\rms}{\ensuremath{\mathrm{s}}}

\newcommand{\veps}{\ensuremath{\varepsilon}}

\newcommand{\vphi}{\ensuremath{\varphi}}

\newcommand{\ol}{\overline}

\newcommand{\eqstop}{\ensuremath{\, \text{.}}}
\newcommand{\eqcomma}{\ensuremath{\, \text{,}}}

\newcommand{\NN}{\ensuremath{\mathbb{N}}}
\newcommand{\ZZ}{\ensuremath{\mathbb{Z}}}
\newcommand{\QQ}{\ensuremath{\mathbb{Q}}}
\newcommand{\RR}{\ensuremath{\mathbb{R}}}
\newcommand{\CC}{\ensuremath{\mathbb{C}}}

\newcommand{\GL}{\operatorname{GL}}

\newcommand{\Cstar}{\ensuremath{\mathrm{C}^*}}

\newcommand{\supp}{\ensuremath{\mathop{\mathrm{supp}}}}
\newcommand{\im}{\ensuremath{\mathop{\mathrm{im}}}}

\newcommand{\Cstarred}{\ensuremath{\Cstar_\mathrm{red}}}

\newcommand{\cont}{\ensuremath{\mathrm{C}}}
\newcommand{\contb}{\ensuremath{\mathrm{C}_\mathrm{b}}}
\newcommand{\conto}{\ensuremath{\mathrm{C}_0}}
\newcommand{\contc}{\ensuremath{\mathrm{C}_\mathrm{c}}}

\newcommand{\Ltwo}{\ensuremath{{\offinterlineskip \mathrm{L} \hskip -0.3ex ^2}}}
\newcommand{\Lone}{\ensuremath{{\offinterlineskip \mathrm{L} \hskip -0.3ex ^1}}}

\newcommand{\grpaction}[1]{\ensuremath{\stackrel{#1}{\curvearrowright}}}

\newcommand{\tubedim}{\operatorname{dim}_{\mathrm{tube}}}
\newcommand{\dimnuc}{\operatorname{dim}_{\mathrm{nuc}}}
\DeclareMathOperator{\interior}{int}
\DeclareMathOperator{\boxen}{tube}
\def\boxint#1{\interior_{\boxen}(#1)}
\def\boxboundary#1{\partial_{\boxen}(#1)}
\DeclareMathOperator{\vol}{vol}

\DeclareMathOperator{\pair}{Pair}

\newcommand{\pairtube}{\mathcal{H}}
\newcommand{\pairtubeopen}{\mathcal{H}^{\mathrm{open}}}

\newcommand{\pairtubeopeni}{\tilde{\mathcal{H}}_i}
\newcommand{\haar}{m}
\newcommand{\lreg}{\lambda}
\newcommand{\nought}{^{(0)}}

\newcommand{\authors}{Ulrik Enstad $\bullet$ Gabriel Favre $\bullet$ Sven Raum}
\renewcommand{\title}{Free actions of polynomial growth Lie groups and classifiable C*-algebras}

\begin{document}


\thispagestyle{empty}

\noindent
\begin{minipage}{\linewidth}
  \begin{center}
    \textbf{\Large \title} \\
    \authors    
  \end{center}
\end{minipage}


\vspace{2em}
\noindent
\begin{minipage}{\linewidth}
  \textbf{Abstract}.
  We show that any free action of a connected Lie group of polynomial growth on a finite dimensional locally compact space has finite tube dimension.  This is shown to imply that the associated crossed product C*-algebra has finite nuclear dimension.  As an application we show that C*-algebras associated with certain aperiodic point sets in connected Lie groups of polynomial growth are classifiable.  Examples include cut-and-project sets constructed from irreducible lattices in products of connected nilpotent Lie groups.
\end{minipage}


\section{Introduction}
\label{sec:introduction}

An important step in Elliott's classification programme for C*-algebras has been the insight that besides long anticipated assumptions such as simplicity and nuclearity, it is necessary to assume some additional regularity of the C*-algebras under consideration. One such notion of regularity is that of finite nuclear dimension, which was introduced by Winter and Zacharias in \cite{winterzacharias2010} as a noncommutative generalisation of finite Lebesgue covering dimension. A C*-algebra is called classifiable if it is simple, separable, non-elementary, unital, satisfies the UCT and has finite nuclear dimension.  Work of many hands has shown that such C*-algebras are classified up to isomorphism by $\rK$-theory and tracial data; see \cite{winter18-icm} and the references therein.

After Elliott's classification programme saw its breakthrough, the investigation of natural examples arising for instance from topological dynamics moved into the focus of substantial parts of the community.  One bridge between these topics is created by the crossed product C*-algebra $\conto(X) \rtimes G$ which can be associated with any action $G \grpaction{} X$ of a locally compact group on a locally compact Hausdorff space.  This line of research can be roughly divided into two directions: One investigating actions of discrete groups and the other one actions of connected groups.  We are not concerned with discrete group actions in this article, but point the reader to the work in \cite{hirshbergwinterzacharias2015, szabowuzacharias2019, naryshkin2022, gardellageffenkranznaryshkin2023}, where classifiability of crossed products with various discrete groups has been established.  In particular, \cite{szabowuzacharias2019} shows finite nuclear dimension of crossed product \mbox{C*-algebras} associated with free actions of finitely generated groups of polynomial growth acting on compact spaces of finite covering dimension.

For actions of connected locally compact groups, classifiability results and nuclear dimension estimates so far have been restricted to $\RR$-actions.  In \cite{hirshbergszabowinterwu2017}, Hirshberg, Szabo, Winter and Wu developed a notion of Rokhlin dimension for $\RR$-actions which was used to show that the crossed product of a free action of $\RR$ on a finite-dimensional space has finite nuclear dimension.  We note that Hirshberg and Wu were later able to remove the assumption that the action is free in \cite{hirshbergwu2021}.

The main theorem of the present paper extends the main result of \cite{hirshbergszabowinterwu2017} to free actions of connected Lie groups of polynomial growth.  By Breuillard's work \cite{breuillard14} the growth rate of any compactly generated group of polynomial growth $G$ behaves asymptotically like $n^{\rmd(G)}$ for some uniquely determined integer $\rmd(G)$.  We prove the following explicit nuclear dimension bound in terms of the dimension $\dim X$ and $\rmd(G)$.
\begin{theorem}
  \label{thm:intro-nuclear-dimension-estimate}
  Let $G \grpaction{} X$ be a free action of a connected Lie group of polynomial growth on a locally compact space $X$. Then
  \begin{gather*}
    \dimnuc(\conto(X) \rtimes G)  \leq 11^{\rmd(G)} \cdot (\dim X + 1)^2 - 1
    \eqstop
  \end{gather*}
\end{theorem}

An important motivation for the present work is the application to groupoid C*-algebras associated with point sets in groups of polynomial growth. In Euclidean spaces, dynamical properties of point sets and tilings are studied in the field of aperiodic order \cite{baakegrimm13} which has been motivated by the discovery of quasicrystals in the 1980s. Bellissard and Kellendonk \cite{bellissard86,bellissard92,kellendonk95} constructed groupoids from such point sets and tilings, which have been studied from the point of view of operator algebras, noncommutative geometry and mathematical physics. Of particular interest to us are model sets (or cut-and-project sets) which were introduced by Meyer \cite{meyer1972}. More recently model sets and other point sets have been introduced and studied in the setting of non-abelian locally compact groups in \cite{bjorklundhartnickpgorzelski2018, bjorklundhartnickpogorzelski2021, bjorklundhartnickpgorzelski2022,hrushovski2012-stable-group-theory, bjorklundhartnick2018, machado2023-higher-rank}. In light of this development, the first and third named author recently extended the definition of the previously studied transversal groupoid of a point set in a Euclidean space to the locally compact group setting in \cite{enstadraum2022}. In particular, to a Delone set $\Lambda$ in a locally compact group $G$ one can associate an {\'e}tale groupoid $\cG(\Lambda)$ which arises as the restriction of the transformation groupoid of the so-called hull dynamical system $G \curvearrowright \Omega(\Lambda)$ to a canonical transversal $\Omega_0(\Lambda)$.  When $\Lambda$ is a discrete subgroup $\Omega_0(\Lambda)$ is a one-point space and the groupoid $\cG(\Lambda)$ agrees with $\Lambda$ as an abstract group.  Hence the groupoid $\Cstar$-algebra $\Cstar(\Lambda) := \Cstar(\cG(\Lambda))$ generalises the group $\Cstar$-algebra of a discrete group.

As $\Cstar(\Lambda)$ is Morita equivalent to a crossed product C*-algebra, the nuclear dimension estimates obtained in \Cref{thm:intro-nuclear-dimension-estimate} can be used to prove that it is classifiable. For an explanation of the terminology we refer the reader to \Cref{sec:classifiable-point-sets}.
\begin{theorem}
  \label{thm:intro-classifiable-point-sets}
  Let $\Lambda$ be a repetitive, aperiodic, FLC Delone set in a connected Lie group of polynomial growth. Then its C*-algebra $\Cstar(\Lambda)$ is classifiable by the Elliott invariant.
\end{theorem}
This theorem generalises the work of Ito-Whittaker-Zacharias on tiling C*-algebras associated with repetitive, aperiodic, FLC tilings \cite{itowhittakerzacharias2019-arxiv-v2}, whose approach established almost finiteness of the associated {\'e}tale groupoids and thus Jiang-Su stability of the tiling C*-algebra.  In contrast, our approach establishes a nuclear dimension bound for the C*-algebra $\Cstar(\Lambda)$, only depending on the Lie group under consideration.

The hypotheses of \Cref{thm:intro-classifiable-point-sets} can be verified in particular for model sets arising from irreducible lattices in products of connected nilpotent Lie groups. We refer the reader to Section~\ref{sec:classifiable-point-sets} for natural examples arising from nilpotent Lie algebras over algebraic number fields.
\begin{corollary}
  \label{cor:intro:irreducible-lattice-classifiable}
  Let $G$ be a connected, simply connected nilpotent Lie group and let $\Lambda \subseteq G$ be a regular model set arising from a cut-and-project scheme $(G, H, \Gamma)$ where $H$ is another nilpotent Lie group and $\Gamma$ is an irreducible lattice.  Then $\Cstar(\Lambda)$ is classifiable.
\end{corollary}

Let us now describe the techniques we use to prove \cref{thm:intro-nuclear-dimension-estimate}.  The results of \cite{hirshbergszabowinterwu2017,hirshbergwu2021} are based on the use of slices. For the purposes of the present paper, a slice for a group action $G \grpaction{} X$ is a compact subset $S$ of $X$ which has the property that for some compact identity neighborhood $K \subseteq G$, the group action $(g,x) \mapsto gx$ is a homeomorphism onto its image when restricted to $K \times S$. The image of this map is called a tube (or a box). This is a local version of slices introduced by Bartels, L{\"u}ck and Reich in \cite{bartelsluckreich2008-covers} as a modification of the global slices introduced in the context of principal bundles, for example by Palais \cite{palais1961-slices}.  For $\RR$-actions on metric spaces, the existence of so-called long, thin covers of tubes was established by Kasprowski and R{\"u}ping in \cite{kasprowskiruping17}.  In \cite{hirshbergwu2021,hirshbergszabowinterwu2017}, these covers are used to construct partitions of unity consisting of Lipschitz functions which are instrumental for the nuclear dimension estimates of the corresponding crossed products.

Our proof follows the strategy to locally trivialise actions by means of tubes as already done in \cite{hirshbergszabowinterwu2017,hirshbergwu2021}. It is however not possible anymore to directly refer to \cite{kasprowskiruping17,bartelsluckreich2008-covers}, where only $\RR$-actions are considered.  Instead we have to develop suitable extensions to connected groups of polynomial growth. Further, the notion of Lipschitz functions is not adapted to non-abelian groups, so that a suitable replacement has to be developed.  In particular we establish the following technical innovations. 
\begin{itemize}
\item We prove in \cref{thm:K-slices-exist-matrix-groups} that every free action of a matrix Lie group $G \grpaction{} X$ admits slices, that is, every point $x \in X$ belongs to the interior of some tube.  This result is extended to connected Lie groups of polynomial growth -- which are not necessarily matrix Lie groups -- in \cref{cor:K-slices-exist-polynomial-growth}.  This step extends the case of $\RR$-actions considered in \cite[Lemma 2.11]{bartelsluckreich2008-covers}.

\item We show in \cref{thm:covering} that if a free action $G \grpaction{} X$ admits slices and $G$ has polynomial growth, then $X$ admits so-called long covers.  That is, for every compact set $K \subseteq G$ there is a cover $\cU$ of $X$ of finite multiplicity (bounded in terms of $\rmd(G)$ and $\dim X$) by interiors of open tubes such that for every $x \in X$, $Kx \subseteq U$ for some $U \in \mathcal{U}$. This implies in particular a bound on the \emph{tube dimension} of $G \grpaction{} X$, a notion we extend from the case of $\RR$-actions in \cite{hirshbergszabowinterwu2017}.
  
\item For $G$ amenable, we prove that a bound on the tube dimension of $G \grpaction{} X$ is equivalent to a number of other statements, cf.\ \Cref{prop:characterisation-tube-dimension}, in particular the existence of partitions of unity consisting of so-called F{\o}lner functions.  This is a modification of the techniques used in \cite{hirshbergszabowinterwu2017,hirshbergwu2021} where Lipschitz partitions of unity are used.
\end{itemize}
Finally, we establish an explicit bound of the nuclear dimension of $\conto(X) \rtimes G$ in terms of the tube dimension of $G \grpaction{} X$ and the dimension of $X$, simplifying some arguments of \cite{hirshbergwu2021} along the way.


\subsection*{Acknowledgements}

The first author acknowledges support from The Research Council of Norway through project 314048. He is also indebted to Stockholm University and the University of Potsdam for their hospitality during extended research stays in 2021--2023. The second and third authors were supported by the Swedish Research Council through grant number 2018-04243.

The third author would like to thank Nigel Higson and Bram Mesland for useful discussions on the K-theory of C*-algebras considered in this work.

\section{Preliminaries}
\label{sec:preliminaries}

In this section we recall some of the objects used throughout the article.  This is also an opportunity to fix notations and conventions. As a general convention, all groups $G$ and spaces $X$ are locally compact Hausdorff unless otherwise specified.

\subsection{Dimension theory}
\label{sec:dimension-theory}

We will need some basic dimension theory for topological spaces. Topological dimension will only be applied to normal, metrisable spaces, and for such spaces the most commonly used notions of dimension coincide. In particular, the frequently encountered Lebesgue covering dimension equals the small inductive dimension, which is well suited to the arguments in the present work. We denote the dimension of such a topological space $X$ unambiguously by $\dim X$.
\begin{definition}
  \label{def:small_inductive_dimension}
  The \emph{small inductive dimension} of a topological space $X$ is defined recursively by the conditions that
  \begin{enumerate}
  \item $\dim(\emptyset) = -1$,
  \item
    \label{it:small-ind-dim:bound}
    $\dim X \leq k$ if for every $x \in X$ and every open neighbourhood $U$ of $x$ there exists an open neighbourhood $U' \subseteq U$ of $x$ such that $\dim(\partial U') \leq k-1$, and
  \item $\dim X = k$ if $\dim X \leq k$ but it is not the case that $\dim X \leq k-1$.
  \end{enumerate}
\end{definition}

\begin{remark}
  \label{rmk:inductive_dimension_open_regular}
  Note that the open neighbourhood $U'$ in \ref{it:small-ind-dim:bound} of \cref{def:small_inductive_dimension} can be chosen to be regular open, that is, $\overline{U'}^{\circ} = U'$, whenever $U$ is regular open. Indeed, if $U$ is regular open then $x \in U'' := \overline{U'}^{\circ} \subseteq \overline{U}^{\circ} = U$, and since $\partial U'' \subseteq \partial U'$ we obtain $\dim(\partial U'') \leq \dim(\partial U') \leq k-1$. Hence $U'$ may be replaced with the regular open set $U''$.
\end{remark}
We will frequently apply the following properties of dimension.
\begin{proposition}
  \label{prop:dim_properties}
  Let $M$ and $N$ be second countable, normal spaces.
  \begin{enumerate}
  \item
    \label{it:dim:monotone}
    Every subspace $M'$ of $M$ satisfies $\dim(M') \leq \dim(M)$.
  \item
    \label{it:dim:union}
    If $M$ is the countable union of a sequence of closed sets $(F_n)_{n \in \bN}$, each with $\dim(F_n) \leq k$, then $\dim(M) \leq k$.
  \item
    \label{it:dim:products}
    $\dim(M \times N) \leq \dim M + \dim N$.
  \item
    \label{it:dim:covers}
    If $\mathcal{U}$ is an open cover of $M$, then $\dim M = \sup_{U \in \mathcal{U}} \dim U$.
  \end{enumerate}
\end{proposition}

\begin{proof}
Parts \ref{it:dim:monotone}, \ref{it:dim:union} and \ref{it:dim:products} of the proposition are respectively 1.1.2, 1.5.4 and 1.5.16 of \cite{engelking78}. To see \ref{it:dim:covers}, note first that $M$ is metrisable and hence paracompact, so we may choose a locally finite refinement $\cV$ of $\cU$. Since open sets are $F_\sigma$ in a metrisable space, Theorem 4.1.10 of \cite{engelking78} gives that $\dim M = \sup_{V \in \cV} \dim V$. Since each $V \in \cV$ is a subset of some $U \in \cU$, \ref{it:dim:monotone} gives $\dim V \leq \dim U$, so $\dim M = \sup_{U \in \cU} \dim U$.
\end{proof}

\subsection{Nuclear dimension}

\begin{definition}
    The \emph{nuclear dimension} of a $\Cstar$-algebra $A$, denoted by $\dimnuc{A}$, is the least number $n$ such that the following hold: There exists a net $(F_\lambda, \Psi_\lambda, \Phi_\lambda)_{\lambda \in \Lambda}$ where $F_\lambda$ are finite-dimensional $\Cstar$-algebras and $\Psi_\lambda \colon A \to F_\lambda$, $\Phi_\lambda \colon F_\lambda \to A$ with the following properties:
    \begin{enumerate}
        \item $(\Phi_\lambda \circ \Psi_\lambda)(a) \to a$ uniformly on finite subsets of $A$,
        \item $\| \Psi_\lambda \| \leq 1$ for all $\lambda$, and
        \item for each $\lambda \in \Lambda$, $F_\lambda$ decomposes into $n+1$ ideals $F_\lambda = F_\lambda^{(0)} \oplus \cdots \oplus F_\lambda^{(n)}$ such that $\Phi_\lambda |_{F_\lambda^{(i)}}$ is a completely positive order zero map for each $0 \leq i \leq n$.
    \end{enumerate}
\end{definition}
The nuclear dimension generalises the dimension of a topological space in as fas as when $X$ is a metrisable locally compact space, $\dimnuc (\mathrm{C}_0(X))$ agrees with the unambiguous topological dimension of $X$ from \Cref{sec:dimension-theory}.

In the proof of our main result we will employ the following lemma which has already been used in \cite{hirshbergwu2017,hirshbergwu2021} to facilitate estimates of nuclear dimension. 
\begin{lemma}
  [See \mbox{\cite[Lemma 1.2]{hirshbergwu2017}}]
  \label{lem:nuclear_lemma}
  Let $B$ be a separable, nuclear $\Cstar$-algebra and let $B_0$ be a dense subset of the unit ball of $B$. Then $\dimnuc(B) \leq n$ if and only if for every finite subset $F \subseteq B_0$ and every $\veps > 0$ there exist a $\Cstar$-algebra $A = \bigoplus_{l=0}^d A^{(l)}$ and completely positive maps $\Phi = \bigoplus_{l=0}^d \Phi^{(l)} \colon B \to A$, $\Psi = \sum_{l=0}^d \Psi^{(l)} \colon A \to B$ such that
  \begin{enumerate}
  \item $\Phi$ is contractive,
  \item $\Psi^{(l)} = \sum_{k=0}^{d^{(l)}} \Psi^{(l,k)}$ where each $\Psi^{(l,k)} \colon A \to B$ is an order zero contraction,
  \item $\| \Psi(\Phi(x)) - x \| < \veps$ for all $x \in F$, and
  \item $\sum_{l=0}^d (\dimnuc(A^{(l)})+1)(d^{(l)} + 1) - 1 \leq n$.
  \end{enumerate}
\end{lemma}

\subsection{Locally compact groupoids and their C*-algebras}
\label{sec:groupoids}

We recall the definition of a topologial groupoid and its C*-algebras here and refer to \cite{renault80} for a comprehensive introduction to the subject.  A \emph{groupoid} is a set $\cG$ together with a distinguished subset $\cG^{(2)} \subseteq \cG \times \cG$, a multiplication map $\cG^{(2)} \to \cG$, $(\alpha,\beta) \mapsto \alpha \beta$, and an inversion map $\cG \to \cG$, $\alpha \mapsto \alpha^{-1}$, such that the following properties are satisfied.
\begin{enumerate}
\item If $(\alpha,\beta),(\beta,\gamma) \in \cG^{(2)}$, then $(\alpha\beta,\gamma),(\alpha,\beta\gamma) \in \cG^{(2)}$, and $(\alpha\beta)\gamma = \alpha(\beta\gamma)$.
\item $(\alpha,\alpha^{-1}) \in \cG^{(2)}$ for all $\alpha \in \cG$, and for all $(\alpha,\beta) \in \cG^{(2)}$ we have that $\alpha^{-1}(\alpha \beta) = \beta$ and $(\alpha\beta)\beta^{-1} = \alpha$.
\item For all $\alpha \in \cG$ we have that $(\alpha^{-1})^{-1} = \alpha$.
\end{enumerate}
The \emph{unit space} of the groupoid is the set $\cG^{(0)} = \{ \alpha^{-1}\alpha \mid \alpha \in \cG \} = \{ \alpha \alpha^{-1} \mid \alpha \in \cG \}$. The \emph{source} and \emph{range} maps $\rms,\rmr \colon \cG \to \cG^{(0)}$ are given by $\rms(\alpha) = \alpha^{-1}\alpha$ and $\rmr(\alpha) = \alpha \alpha^{-1}$ for $\alpha \in \cG$. We have that
\begin{gather*}
  \cG^{(2)} = \{ (\alpha,\beta) \in \cG \mid \rms(\alpha) = \rmr(\beta) \}
  \eqstop
\end{gather*}
For each $x \in \cG^{(0)}$ we set
\begin{gather*}
  \cG_x = \{ \alpha \in \cG \mid \rms(\alpha) = x \}\eqcomma
  \qquad \text{and} \quad
  \cG^x = \{ \alpha \in \cG \mid \rmr(\alpha) = x \}
  \eqstop
\end{gather*}
A \emph{locally compact groupoid} is a groupoid $\cG$ equipped with a locally compact Hausdorff topology under which the multiplication, inversion, source and range maps are continuous with respect to the relative topologies on $\cG^{(2)}$ and $\cG^{(0)}$. We will assume that all locally compact groupoids are Hausdorff and that range and source maps are open. We say that $\cG$ is \emph{{\'e}tale} if the source map (equivalently range map) is a local homeomorphism when viewed as a map $\cG \to \cG$.

Given a subset $X$ of the unit space $\cG^{(0)}$, the \emph{restriction of $\cG$ to $X$}, denoted by $\cG|_X$, is the subset $\{ \alpha \in \cG \mid \rms(\alpha),\rmr(\alpha) \in X \}$ of $\cG$. This is a groupoid with unit space $X$ where the operations are restricted from $\cG$.

A \emph{Haar system} for a locally compact groupoid $\cG$ is a collection $(\haar^x)_{x \in \cG^{(0)}}$ where $\haar^x$ is a positive regular Borel mesure on $\cG^x$ for each $x \in \cG^{(0)}$, such that the following requirements are satisfied:
\begin{enumerate}
\item The support of $\haar^x$ is equal to $\cG^x$ for each $x \in \cG^{(0)}$.
\item For every $g \in \contc(\cG)$ the function $x \mapsto \int_{\cG^{x}} g(\alpha) \, \rmd{\haar^x(\alpha)}$ belongs to $\contc(\cG^{(0)})$.
\item For every $\alpha \in \cG$ and $g \in \contc(\cG)$ we have that
  \begin{gather*}
    \int_{\cG^{\rms(\alpha)}} f(\alpha\beta) \, \rmd{\haar^{\rms(\alpha)}}(\beta)
    =
    \int_{\cG^{\rmr(\alpha)}} f(\beta) \, \rmd{\haar^{\rmr(\alpha)}(\beta)}
    \eqstop
  \end{gather*}
\end{enumerate}
Given a Haar system $(\haar^x)_{x \in \cG^{(0)}}$ on $\cG$, one defines the convolution of two functions $f,g \in \contc(\cG)$ by the formula
\begin{gather*}
  (f*g)(\alpha)
  =
  \int_{\cG^x} \sigma(\beta,\beta^{-1}\alpha) f(\beta)g(\beta^{-1}\alpha) \, \rmd \haar^{\rmr(\alpha)}(\beta)
  \eqcomma \quad \text{ for all } \alpha \in \cG
  \eqstop
\end{gather*}
One frequently writes the convolution product implicitly.

With respect to the convolution product and the involution given by $f^*(\alpha) = f(\alpha^{-1})$, the space $\contc(G)$ becomes a $*$-algebra. The $I$-norm on $\contc(\cG)$ is given by
\begin{align*}
  \| f \|_I
  =
  \max \Big\{
  \sup_{x \in \cG^{(0)}} \int_{\cG^x}|f(\alpha)| \, \rmd \haar^x(\alpha)\eqcomma
  \sup_{x \in \cG^{(0)}} \int_{\cG^x} |f(\alpha^{-1})| \, \rmd \haar^x(\alpha)
  \Big\}\eqcomma \quad f \in \contc(\cG)\eqstop
\end{align*}
Any $x \in \cG^{(0)}$ determines a corresponding left regular representation $\lreg_x$ of $\contc(\cG)$ on $\Ltwo(\cG^x,\haar^x)$ given by
\begin{gather*}
  \lreg_x(f)\xi = f * \xi, \quad \text{ for all } f \in \contc(\cG), \xi \in \Ltwo(\cG^x,\haar^x)
  \eqstop
\end{gather*}

The \emph{full C*-algebra} of $\cG$, denoted by $\Cstar(\cG)$ (with respect to the given Haar system) is the \mbox{C*-envelope} of $\contc(G)$, while the \emph{reduced C*-algebra}, denoted by $\Cstarred(\cG)$, is the completion of $\contc(G)$ with respect to the norm
\begin{gather*}
  \| f \|_{\mathrm{red}} = \sup_{x \in \cG^{(0)}}\| \lreg_x(f) \|, \qquad f \in \contc(\cG)
  \eqstop
\end{gather*}
All groupoids considered in this work are amenable, so that their full and reduced groupoid \mbox{C*-algebras} are naturally isomorphic.  See for example \cite[Theorem 6.1.4]{anantharamandelarocherenault2000}.

We have the norm inequality
\begin{equation}
  \| f \|_{\mathrm{red}} \leq \| f \|_I, \qquad \text{for all } f \in \contc(\cG)
  \eqstop
\end{equation}
There is an embedding of $\conto(\cG^{(0)})$ into the multiplier algebra $\rM(\Cstarred(\cG))$ of $\Cstarred(\cG)$ such that if $f \in \conto(\cG^{(0)})$ and $a \in \contc(\cG)$ then $fa, af \in \contc(\cG)$, with
\begin{align}
  (fa)(\alpha) &= f(\rmr(\alpha))a(\alpha)\eqcomma \text{ and}\\
  (af)(\alpha) &= a(\alpha)f(\rms(\alpha))\eqstop
\end{align}
For $a \in \Cstarred(\cG)$, the element $faf$ is called the \emph{compression} of $a$ by $f$.  It follows in particular from the above equations that if $f \in \contc(\cG^{(0)})$ is supported in an open subspace $X \subseteq \cG^{(0)}$ and $a \in \Cstarred(\cG)$, then $faf \in \Cstarred(\cG|_X)$.

The following examples will be important later.
\begin{example}[Transformation groupoids]
  \label{ex:transformation_groupoid}
  Let $G$ be a locally compact (Hausdorff) group acting (on the left) on a locally compact Hausdorff space $X$.  We can then associate to this data a corresponding groupoid $\cG = G \ltimes X$, called a \emph{transformation groupoid}.  As a topological space we have $\cG = G \times X$, and the unit space is given by $\{ (e,x) \mid x \in X \} \cong X$. The source and range maps are given by $\rms(g,x) = x$ and $\rmr(g,x) = gx$, the multiplication is given by $(h,gx)(g,x) = (hg,x)$ and the inversion is given by $(g,x)^{-1} = (g^{-1},gx)$.  A natural Haar system on this groupoid comes from the Haar measure $\haar$ on $G$.  Since $\cG^x = \{ (g,g^{-1}x) \mid g \in G \} \cong G$, we can define $\haar^x$ via
  \begin{gather*}
    \int_{\cG^x} f(\alpha) \, \rmd \haar^x(\alpha)
    =
    \int_G f(g,g^{-1}x) \, \rmd \haar(g)\eqcomma \qquad \text{for all } f \in \contc(\cG)
    \eqstop
  \end{gather*}

The full and reduced groupoid C*-algebras of $G \ltimes X$ are isomorphic to the full and reduced crossed products $\conto(X) \rtimes G$ and $\conto(X) \rtimes_{\mathrm{red}} G$ associated to $G \grpaction{} X$, respectively.  For amenable groups $G$, which will be exclusively considered in this work, the transformation groupoid $G \ltimes X$ is amenable by \cite[Example 2.2.14 (1) in connection with Corollary 2.2.10]{anantharamandelarocherenault2000}. Hence also the reduced and the full crossed products are isomorphic.
  
\end{example}

\begin{example}[Pair groupoids]
  \label{ex:pair_groupoid}
  Let $X$ be a locally compact Hausdorff space. The \emph{pair groupoid} $\pair(X)$ of $X$ is defined as a set $\pair(X) = X \times X$ with unit space $X \cong \{ (x,x) \mid x \in X \} \subseteq X \times X$ and with source and range maps given by $\rms(x,y) = y$ and $\rmr(x,y) = x$, respectively. Its multiplication is given by
  \begin{gather*}
    (x,y)(y,z) = (x,z)\eqcomma \quad \text{for all } x,y,z \in X
    \eqstop
  \end{gather*}
  Haar systems on $\cG$ are in one-to-one correspondence with regular Borel measures $\mu$ on $X$ with full support, and the reduced groupoid C*-algebra of $\pair(X)$ is isomorphic to $\mathcal{K}(\Ltwo(X,\mu))$, the compact operators on $\Ltwo(X,\mu)$, see e.g.\  \cite[Theorem 3.1.2]{paterson1999}. When $\Ltwo(X,\mu)$ is separable (e.g.\ when $\mu$ is $\sigma$-finite) then $\mathcal{K}(\Ltwo(X,\mu))$ has nuclear dimension equal to zero.
\end{example}

\begin{example}[Topological spaces]
  \label{ex:top_space_groupoid}
  A degenerate example of a locally groupoid occurs when $\cG = \cG^{(0)}$. In this case $\cG^{(2)} = \cG^{(0)} \times \cG^{(0)}$, with multiplication and inversion trivially given by $x \cdot x = x$ and $x^{-1} = x$ for $x \in \cG$. Conversely, any locally compact space $X$ can be equipped with this groupoid structure. For each $x \in \cG^{(0)}$ we have that $\cG^x = \{ x \}$, so a natural Haar system is given by assigning unit mass to $x$. The corresponding convolution and involution collapse to pointwise operations, and the associated C*-algebras are isomorphic to $\conto(X)$, the continuous functions vanishing at infinity on $X$.
\end{example}

\begin{example}[Product groupoids]
    \label{ex:product_groupoids}
    If $\cG$ and $\cH$ are locally compact groupoids, the Cartesian product $\cG \times \cH$ may be given a locally compact groupoid structure in the evident way. A pair of Haar systems for $\cG$ and $\cH$ gives rise to a Haar system for $\cG \times \cH$ via the corresponding product measures. The full (reduced) groupoid $\Cstar$-algebra of $\cG \times \cH$ is isomorphic to the maximal (spatial) tensor product of the full (reduced) $\Cstar$-algebras of $\cG$ and $\cH$.
\end{example}

\subsection{Polynomial growth}
\label{sec:polynomial-growth}

A locally compact group $G$ is said to be \emph{compactly generated} if it admits a compact generating set, that is, a compact subset $\Omega \subseteq G$ such that $G = \bigcup_{n=0}^\infty (\Omega \cup \Omega^{-1})^n$. We will usually assume $\Omega$ to be symmetric, that is, $\Omega^{-1} = \Omega$. Furthermore, $G$ is said to be of \emph{polynomial growth} if there exists $d \in \NN$ and $C > 0$ such that $m(\Omega^n) \leq C n^d$ for all $n \in \bN$. The definition can be shown to be independent of the chosen generating set. By a fundamental result of Breuillard \cite{breuillard14} there exists an integer $\rmd(G) \geq 0$ depending only on $G$ and a constant $c(\Omega) > 0$ (depending on the normalisation of a Haar measure $\haar$) such that
\begin{equation}
  \label{eq:breuillard}
  \lim_{n \to \infty} \frac{\vol_G(\Omega^n)}{n^{\rmd(G)}} = c(\Omega) \eqstop
\end{equation}

We will need the following covering property for generating sets in polynomial growth groups. The proof is a standard maximal packing argument that we include here for completeness.
\begin{proposition}
  \label{prop:cover_by_translates}
    Suppose $G$ is compactly generated of polynomial growth, say with compact symmetric generating set $\Omega \subseteq G$, and let $a \in \NN$. Then there exists an $N \in \NN$ such that for all $n \geq N$, any translate of $\Omega^{an}$ can be covered by $(a+1)^{\rmd(G)}$ translates of $\Omega^{2n}$.
\end{proposition}
\begin{proof}
  Let $\veps > 1$ be such that $\veps(a + 1)^{\rmd(G)} < (a + 1)^{\rmd(G)} + 1$, fix a Haar measure $\haar$ on $G$ with the associated constant $c(\Omega)$ and set $\delta = c(\Omega)\frac{\veps -1}{\veps+1} > 0$.  By formula \eqref{eq:breuillard} there exists $N \in \NN$ such that
  \begin{equation}
     \label{eq:inequalities_breuillard}
     (c(\Omega) - \delta) n^{\rmd(G)} \leq \haar(\Omega^n) \leq (c(\Omega) + \delta) n^{\rmd(G)}
    \eqcomma \qquad \text{ for all } n \geq N \eqstop
  \end{equation}
  Let $n \geq N$ and let $\cS$ be the set of subsets $S \subseteq \Omega^{an}$ such that $g^{-1}g' \notin \Omega^{2n}$ for all $g,g' \in S$ with $g \neq g'$. Note that for every $S \in \cS$ the sets $\{ g \Omega^{n} \mid g \in S \}$ are pairwise disjoint.  Indeed, if $h \in g \Omega^{n} \cap g' \Omega^{n}$ for distinct $g,g' \in S$ then $g^{-1}g' = (h^{-1}g)^{-1}(h^{-1}g') \in \Omega^{n}\Omega^{n} = \Omega^{2n}$, which is a contradiction.

  We order $\cS$ with respect to inclusion. Then every totally ordered subset of $\cS$ has an upper bound (namely the union), hence $\cS$ admits a maximal element $S$ by Zorn's lemma. We claim that $\Omega^{an} \subseteq \bigcup_{g \in S} g \Omega^{2n}$ holds.  Otherwise we can find $h \in \Omega^{an}$ with $h \notin g\Omega^{2n}$ for any $g \in S$, which implies that $S \cup \{ h \} \in \cS$. This contradicts the maximality of $S$, so we conclude that $\Omega^{an}$ can be covered by translates of $\Omega^{2n}$ by elements from $S$.

  It remains to bound the size of $S$. Using the pairwise disjointness of $\{ g\Omega^n \mid g \in S \}$ and the inclusion $\bigcup_{g \in S} g \Omega^{n} \subseteq \Omega^{an}\Omega^n = \Omega^{(a+1)n}$ we have that
  \begin{gather*}
    |S| \haar(\Omega^{n}) = \haar \Big( \bigcup_{g \in S} g \Omega^{n} \Big) \leq \haar(\Omega^{(a+1)n})
    \eqstop
  \end{gather*}

  Applying both bounds in \eqref{eq:inequalities_breuillard} gives
  \begin{gather*}
    |S|
    \leq
    \frac{\haar(\Omega^{(a+1)n})}{\haar(\Omega^{n})}
    \leq
    \frac{(c(\Omega) + \delta)((a+1)n)^{\rmd(G)}}{(c(\Omega) - \delta)n^{\rmd(G)}}
    =
    \veps \cdot (a+1)^{\rmd(G)} < (a+1)^{\rmd(G)} + 1
    \eqcomma
  \end{gather*}
  by the choice of $\veps$.  This implies that $|S| \leq (a+1)^{\rmd(G)}$.
\end{proof}


\section{Slices}
\label{sec:slices}

Slices are used to understand in how far actions of locally compact groups are locally trivial.  In the relevant context of non-compact groups, the concept was introduced by Palais' for proper actions \cite{palais1961-slices}.  In this article we are concerned with free actions, for which we can show the existence of slices for suitable classes of groups in \cref{sec:slices-existence}.  The following definition of a slice is suitable for this setup, while for non-free actions a definition closer to Palais' work is needed.
\begin{definition}
  \label{def:slice}
  Let $G \grpaction{} X$ be an action.  Let $K \subseteq G$ be some compact identity neighbourhood. A \emph{$K$-slice} is a subset $S \subseteq X$ such that the map $K \times S \to X$ given by $(g,x) \mapsto gx$ is injective. The resulting image $KS$ in  $X$ is called a \emph{tube}, and its interior $(KS)^{\circ}$ is called an \emph{open tube}.
\end{definition}

\begin{remark}
  \label{rem:slice-homeomorphic}
  The injectivity of the map $K \times S \to X$ in \Cref{def:slice} implies that it is a homeomorphism onto its image since $K \times S$ is compact and $X$ is Hausdorff. Therefore the tube $KS$ is a compact set in $X$ homeomorphic to $K \times S$.
\end{remark}

\begin{remark}
  \label{rem:slice-disjoint-translates}
  We will frequently use the following facts.  Any closed subset of a $K$-slice is also a $K$-slice, and if $L \subseteq K$ is a compact identity neighbourhood, then every $K$-slice is also an $L$-slice.  Furthermore, $S$ is a $K$-slice if and only if $gS \cap S = \emptyset$ for all $g \in K^{-1}K \setminus \{ e \}$.
\end{remark}

\subsection{Properties of slices}
\label{sec:slices-properties}

In this section we collect some basic properties of slices that will be used in the remainder of the article.  We start by considering a suitable notion of the interior of a slice.
\begin{lemma}
  \label{lem:box-interior-independence}
  Let $K$ and $K'$ be compact identity neighborhoods in $G$ such that the set $S \subseteq X$ is both a $K$-slice and a $K'$-slice. Then
  \begin{gather*}
    S \cap (KS)^{\circ} = S \cap (K'S)^{\circ}
    \eqstop
  \end{gather*}
\end{lemma}
\begin{proof}
  Passing to the intersection $K \cap K'$, we may assume that $K \subseteq K'$ and prove that $S \cap (K'S)^{\circ} \subseteq S \cap (KS)^{\circ}$.  Let $x \in S \cap (K'S)^{\circ}$.  Since $(K'S)^{\circ} \subset K'S$ is open and products of open subsets form a basis of the topology of $K' \times S$, there are open subsets $U \subseteq K'$ and $V \subseteq S$ such that $x \in UV \subseteq (K'S)^\circ$.  Since $x \in S$, we find that $e \in U$ and may hence assume without loss of generality that $U \subseteq K$.  Then $x \in UV \subseteq KS$ on the one hand and on the other hand we obtain a inclusions where each term is open in the next one $UV \subseteq (K'S)^\circ \subseteq X$.  So $UV \subseteq X$ is open, which shows that $x \in (KS)^\circ$. 
\end{proof}

\begin{definition}\label{def:box-interior}
  Given a $K$-slice $S$, we define its \emph{tube interior} and its \emph{tube boundary} to be the sets
  \begin{align*}
    \boxint{S} &= S \cap (KS)^{\circ}, \\
    \boxboundary{S} &= S \setminus \boxint{S} ,
  \end{align*}
  respectively.
\end{definition}
By \cref{lem:box-interior-independence} we infer that $\boxint{S}$ and $\boxboundary{S}$ are independent of the compact identity neighborhood $K$.

We will next obtain some basic properties of the tube interior.  For this, the following result is instrumental, which allows to enlarge a tube in the direction of the group.
\begin{lemma}
  \label{lem:extend_slice}
  Assume that $G$ is second-countable. Let $S$ be a $K$-slice for a compact identity neighborhood $K$ in $G$. Then there exists a compact identity neighborhood $L$ with $K \subseteq \mathring{L}$ such that $S$ is an $L$-slice.
\end{lemma}
\begin{proof}
  Since $G$ is second-countable we can multiply $K$ with elements from a descending sequence of compact neighbourhoods of the identity to find a descending sequence $(L_n)_{n \in \NN}$ of compact sets such that $\mathring{L_n} \supseteq K$ for all $n \in \NN$ and $\bigcap_{n \in \NN} L_n = K$.  Suppose for a contradiction that $S$ is not an $L_n$-slice for any $n \in \NN$. By \Cref{rem:slice-disjoint-translates} we then find $g_n \in L_n^{-1}L_n \setminus K^{-1}K$ and $x_n,x_n' \in S$ such that $g_nx_n = x_n'$ for all $n \in \NN$. Passing to convergent subsequences, we may assume that $x_n \to x$, $x_n' \to x'$ with $x,x' \in S$ and
  \begin{gather*}
    g_n \to g \in \Big( \bigcap_{n \in \NN} L_n^{-1}L_n \Big) \setminus (K^{-1}K)^{\circ} = \partial(K^{-1}K) .
  \end{gather*}
  In particular, $g \neq e$. Continuity of the action then implies $gx' = x$, so $g S \cap S \neq \emptyset$. By \Cref{rem:slice-disjoint-translates} this is a contradiction to the fact that $S$ is a $K$-slice.
\end{proof}
\begin{lemma}
  \label{lem:box_interior}
  Let $K$ be a compact identity neighborhood and let $S$ be a $K$-slice. Then the following identities hold.
  \begin{align}
    (KS)^{\circ} & = \mathring{K} (\boxint{S}) \eqcomma
                   \label{eq:box_interior} \\
    \partial_X(KS) & = (\partial_G K)S \cup K(\boxboundary{S}) \eqstop
                     \label{eq:box_boundary}
  \end{align}
  Furthermore, if $B$ is regular open in $S$ and contained in $\boxint{S}$, then
  \begin{align}
    \boxint{\overline{B}} & = B \eqcomma
                            \label{eq:boxint_open} \\
    \boxboundary{\overline{B}} & = \partial_S B \eqcomma
                                 \label{eq:boxboundary_open} \\
    \partial(K\overline{B}) & = (K\overline{B}) \setminus \mathring{K}B = (\partial_G K)\overline{B} \cup K(\partial_S B) \eqstop
                              \label{eq:several-boundary-expressions}
  \end{align}
\end{lemma}
\begin{proof}
  We start by proving \eqref{eq:box_interior} and first show that $\mathring{K}\boxint{S} \subseteq (KS)^{\circ}$.  Let $x \in \mathring{K}\boxint{S}$, say $x=gy$ where $g \in \mathring{K}$ and $y \in \boxint{S}$.  Set $L = K \cap g^{-1}K$ and consider the open set $U = g(LS)^{\circ}$ in $X$.  Then $U \subseteq gLS \subseteq gg^{-1}KS = KS$, so that $U \subseteq (KS)^{\circ}$ follows.  Furthermore, since $L \subseteq K$, we know that $S$ is also an $L$-slice and \cref{lem:box_interior} shows that $y \in \boxint{S} = S \cap (LS)^{\circ}$. Thus $x = gy \in g(LS)^{\circ} = U \subseteq (KS)^{\circ}$.

  Now we show that $(KS)^{\circ} \subseteq \mathring{K}\boxint{S}$.  Let $x \in (KS)^{\circ}$, say $x = gy$ with $g \in K$ and $y \in S$.  Then $g^{-1}K$ is a compact identity neighbourhood and $S$ is a $g^{-1}K$-slice, so by \cref{lem:box_interior} we obtain
  \begin{gather*}
    y
    =
    g^{-1}x \in S \cap g^{-1}(KS)^{\circ}
    =
    S \cap (g^{-1}KS)^{\circ}
    =
    \boxint{S}
    \eqstop
  \end{gather*}
  It remains to argue that $g \in \mathring{K}$.  For this, we apply \cref{lem:extend_slice} to find a compact identity neighborhood $L$ in $G$ with $K \subseteq \mathring{L}$ such that $S$ is an $L$-slice.  By continuity of the group action, the set
  \begin{gather*}
    W = \mathring{L} \cap \{ h \in G \mid hy \in (KS)^{\circ} \}
  \end{gather*}
  is open in $G$ and it is clear that $g \in W$.  We show that $W \subseteq K$.  Indeed, if $h \in W$ then $hy \in (KS)^{\circ}$ so we can write $hy = h'y'$ with $h' \in K$ and $y' \in S$.  Since $h \in L$ and $S$ is an $L$-slice, this forces $h = h'$ and $y = y'$. In particular, we infer that $h \in K$.
  
  We next prove \eqref{eq:box_boundary}.  First observe that
  \begin{gather*}
    \mathring{K}\boxint{S} = \mathring{K}S \cap K\boxint{S}
    \eqstop
  \end{gather*}
  Indeed, the inclusion form left to right is obvious. Conversely, if $x = gy \in \mathring{K}S \cap K\boxint{S}$ for uniquely determined elements $g \in K$ and $y \in S$, it follows from $x \in \mathring{K}S$ that $g \in \mathring{K}$ and from $x \in K\boxint{S}$ that $y \in \boxint{S}$.  Hence $x \in \mathring{K}\boxint{S}$.  

  Applying the equality above and using \eqref{eq:box_interior}, we get
  \begin{align*}
    (\partial_G K)S \cup K(\boxboundary{S})
    & =
      (K\setminus \mathring{K})S \cup K(S\setminus \boxint{S}) \\
    & =
      (KS\setminus \mathring{K}S) \cup (KS\setminus  K\boxint{S})) \\
    & =
      KS\setminus (\mathring{K}S \cap K\boxint{S})) \\
    & =
      KS \setminus \mathring{K}\boxint{S} \\
    & =
      KS \setminus (KS)^{\circ} \\
    & =
      \partial_X (KS)
      \eqstop
  \end{align*}

  Let us next show \eqref{eq:boxint_open}. First we show that $B \subseteq \boxint{\ol{B}}$. Using that $B \subseteq \boxint{S}$ and \eqref{eq:box_interior} we obtain $\mathring{K}B \subseteq \mathring{K}\boxint{S} = (KS)^{\circ}$.  Since $\mathring{K}B$ is open in $KS$ and hence in $(KS)^{\circ}$, it is also open in $X$.  But then $B \subseteq \ol{B} \cap (K\ol{B})^{\circ} = \boxint{\ol{B}}$.

  Now we show that $\boxint{\ol{B}}\subseteq B$. Let $x \in \boxint{\ol{B}}$. Since $(K\ol{B})^{\circ}$ is open in $KS$ and $x \in (K\ol{B})^{\circ}$, we can by definition of tubes find open sets $W \subseteq K$ and $V \subseteq S$ such that $x \in WV \subseteq (K\ol{B})^{\circ}$. The inclusion $WV \subseteq K\ol{B}$ then forces $V \subseteq \ol{B}$, and since $x \in \ol{B} \subseteq S$ we get $x \in V$. Since $x \in V \subseteq \ol{B}$ and $V$ is open in $S$, this means that $x$ lies in the interior of $\ol{B}$ inside $S$.  By assumption the latter equals $B$, so $x \in B$.

  To show \eqref{eq:boxboundary_open}, we use \eqref{eq:boxint_open} and obtain
  \begin{gather*}
    \boxboundary{\ol{B}}
    =
    \ol{B}\setminus\boxint{\ol{B}}
    =
    \ol{B}\setminus B
    =
    \partial_S{B}
    \eqstop
  \end{gather*}

  Let us now show \eqref{eq:several-boundary-expressions}.  Since $\ol{B}$ is a compact subset of $S$, it is a $K$-slice. Hence \eqref{eq:box_boundary} together with \eqref{eq:boxboundary_open} gives that
  \begin{gather*}
    \partial(K\ol{B})
    =
    (\partial_G K)\ol{B} \cup K(\boxboundary{B})
    =
    (\partial_G K)\ol{B} \cup K(\partial_S B)
    \eqstop
  \end{gather*}

  Moreover, the definition of boundary together with \eqref{eq:box_interior} and \eqref{eq:boxboundary_open} gives that
  \begin{gather*}
    \partial(K\ol{B})
    =
    K\ol{B} \setminus (K\ol{B})^{\circ}
    =
    K\ol{B} \setminus K\boxint{\ol{B}}
    =
    K\ol{B} \setminus KB
    \eqstop
  \end{gather*}
  This finishes the proof.
\end{proof}
The next lemma will be frequently used to extract information about the intersection of tubes and slices.
\begin{lemma}
  \label{lem:slice_restriction_injective}
  If $S$ and $S'$ are $K$-slices, then the restriction of the second factor projection $KS \cong K \times S \to S$ to the set $KS \cap S'$ is a homeomorphism onto its image.
\end{lemma}
\begin{proof}
  Since $KS \cap S'$ is compact and $S$ is Hausdorff, it suffices to show that the map is injective. Let $g_1,g_2 \in K$ and $x_1,x_2 \in S$ such that $g_1x_1 , g_2x_2 \in S'$. Suppose that the images of $g_1x_1$ and $g_2x_2$ under the projection $KS \to S$ are equal, that is, $x_1 = x_2$.  We then have that $g_1^{-1}(g_1x_1) = g_2^{-1}(g_2x_2) \in KS'$ and since $S'$ is a $K$-slice, it follows that $g_1^{-1} = g_2^{-1}$. This implies injectivity.
\end{proof}

\subsection{Existence of slices}
\label{sec:slices-existence}

The existence of $G$-slices for actions of $G$ is a classical topic closely related to the early development of the theory of principal bundles and the question about their local triviality.  The ultimate existence result in this direction was obtained by Palais in \cite{palais1961-slices}.  The notion of $K$-slices and thus tubes as used in the present work goes back to the work of Bartels--L{\"u}ck--Reich in which an equivariant version for $\RR$-actions equipped with an additional commuting action of a discrete group was obtained \cite[Lemma 2.11]{bartelsluckreich2008-covers} based on ideas from Palais' work.  For our needs, this approach needs further refinement in order to treat actions of Lie groups and simplifies at the same time since we target free actions and there is no additional action of a discrete group present. The goal of the this subsection will be to show that actions of matrix Lie groups and of connected Lie groups of polynomial growth admit slices in the following sense.
\begin{definition}
  \label{def:admits_slices}
  We say that an action $G \curvearrowright X$ \emph{admits slices} if for every compact identity neighbourhood $K \subseteq G$ and every $x \in X$ there exists a $K$-slice $S \subseteq X$ such that $x \in \boxint{S}$.
\end{definition}

\begin{remark}
Notice that if an action $G \curvearrowright X$ admits slices then it must be free. Indeed, if $gx =x$ for $g \in G$ and $x \in X$, pick some compact identity neighborhood $K \subseteq G$ that contains $g$. If $S$ is a $K$-slice that contains $x$, then the equation $gx = x = ex$ forces $g=e$ by the definition of a $K$-slice.
\end{remark}
The next lemmas provide us with a sufficiently equivariant map from a free $G$-space into a finite dimensional representation, in analogy with Palais' result \cite[Theorem 1.2.7]{palais1961-slices} in the context of proper $G$-spaces.
\begin{lemma}
  \label{lem:large-identity-neighbourhoods}
  Let $K \subseteq G$ be a compact subset.  Then there is a symmetric relatively compact identity neighbourhood $U \subseteq G$ such that $\bigcap_{k \in K} U k$ is an identity neighbourhood.
\end{lemma}
\begin{proof}
  Replacing $K$ by $K \cup K^{-1}$, we may assume that $K$ is symmetric.  Take any symmetric, relatively compact identity neighbourhood $V \subseteq G$ and define $U = VK \cup KV$.  Then $U$ is symmetric and since $K$ is symmetric is follows that
  \begin{gather*}
    \bigcap_{k \in K} U k
    \supseteq
    \bigcap_{k \in K} \left ( \bigcup_{k' \in K} V k' \right) k
    \supseteq V
    \eqstop
  \end{gather*}
  Since $V$ was relatively compact, also $U$ is relatively compact.
\end{proof}

\begin{lemma}
  \label{lem:map-to-vectorspace}
  Let $G \grpaction{} X$ be a free action, $V$ a finite dimensional $G$-representation and $v \in V$ a vector with trivial stabiliser in $G$.  For every $x \in X$, every neighbourhood $A$ of $x$ and every relatively compact identity neighbourhood $U \subseteq G$ there is a function $f \in \contc(A, \RR)$ such that $\int_U f(g^{-1}x) gv \, \rmd g= v$.
\end{lemma}
\begin{proof}
  Consider the linear map $T \colon \contc(A, \RR) \to V$ given by $Tf = \int_U f(g^{-1}x) g v \, \rmd g$.  Since $V$ is finite dimensional and $\im T$ is a vector subspace of $V$, it suffices to show that for all convex neighbourhoods $v \in C$ we have $C \cap \im T \neq \emptyset$.

  Fix a convex neighbourhood $C$ of $v$ as above and let $U_0 \subseteq G$ be an identity neighbourhood satisfying $U_0 v \subseteq C$.  Since $\ol{U}x \subseteq X$ is closed, there is a neighbourhood $B \subseteq A$ of $x$ satisfying $B \cap \ol{U}x \subseteq U_0x$.  By freeness of $G \grpaction{} X$, this implies that for every $g \in U$, the statement $g x \in B$ implies $g \in U_0$.  Let $f \in \contc(A, \RR)$ satisfy $0 \leq f \leq 1$, $f(x) \neq 0$, $\int_U f(g^{-1}x) \, \rmd g = 1$ and $\supp f \subseteq B$. Then by choice of $B$, the map $g \mapsto \mathbb{1}_U(g) f( g^{-1} x)$ is a probability density on $G$ supported in $U_0$.  Since $U_0v \subseteq C$, this implies that $\int_U f(g^{-1}x) gv \, \rmd g \in C$.
\end{proof}
We are now prepared to constructed sufficiently equivariant maps into finite dimensional representations.
\begin{lemma}
  \label{lem:map-to-vectorspace-equivariant}
  Let $G \grpaction{} X$ be a free action  and assume that $X$ is second-countable. Let $V$ be a finite dimensional $G$-representation and $v \in V$ a vector with trivial stabiliser in $G$.  For every relatively compact identity neighbourhood $K \subseteq G$ and for every $x_0 \in X$ there is a neighbourhood $A$ of $x_0$ and a continuous map $F\colon X \to V$ such that $F(x_0) = v$ and $F(k x) = k \vphi(x)$ for all $k \in K$ and all $x \in A$.
\end{lemma}
\begin{proof}
  By \cref{lem:large-identity-neighbourhoods} there is a symmetric, relatively compact identity neighbourhood $U$ in $G$ such that $U_0 = \bigcap_{k \in K} Uk$ is an identity neighbourhood.  Since $G$ acts freely, we have $x_0 \notin (K \ol{U} \setminus U_0)x_0$.  Hence,
  \begin{gather*}
    \bigcap_{\substack{A \text{ compact}\\ \text{neighbourhood of }x_0}} A \cap (K \ol{U} \setminus U_0)A = \emptyset
    \eqstop
  \end{gather*}
  By compactness, there is a neighbourhood $A$ of $x_0$ such that for any $g \in K \ol{U} \setminus U_0$ and every $x \in A$ we have $gx \notin A$.  Fix such $A$.

  By \cref{lem:map-to-vectorspace}, there is $f \in \contc(A, \RR)$ supported in $A$ such that $\int_U f(g^{-1}x_0) gv \, \rmd g = v$.  Put
  \begin{gather*}
    F(x) = \int_U f(g^{-1}x) gv \, \rmd g
    \eqstop
  \end{gather*}
  For $k \in G$, $x \in X$, left-invariance of the Haar measures shows that
  \begin{align*}
    F(kx)
    & = \int_U f(g^{-1}k^{-1}x) gv \, \rmd g \\
    & = \int_{k^{-1}U} f(g^{-1}x) kgv \, \rmd g \\
    & = k F(x) + \int_{k^{-1}U \setminus U} f(g^{-1}x) kgv \, \rmd g - \int_{U \setminus k^{-1}U} f(g^{-1}x) kgv \, \rmd g
    \eqstop
  \end{align*}
  Let us simplify the last expression.  Recall that $\supp f \subseteq A$ and that the width of $A$ is at most $U_0 = \bigcap_{k \in K} Uk$.  We find that for $x \in A$ and $g \in k^{-1}U \setminus U$, we have in particular $g^{-1} \in KU \setminus U_0$ and hence $g^{-1}x \notin A$.  Similarly, for $g \in U \setminus k^{-1}U$ we have $g^{-1} \notin Uk$, so that $g^{-1}x \notin A$ follows as above.  We infer that $F(kx) = kF(x)$ for all $k \in K$ and all $x \in A$.  
\end{proof}
We are now able to prove the existence of slices in the sense of \Cref{def:admits_slices} for free actions of matrix Lie groups, that is, subgroups of $\GL(n,\RR)$ for some $n \in \bN$.
\begin{theorem}
  \label{thm:K-slices-exist-matrix-groups}
  Let $G$ be a matrix Lie group and $G \grpaction{} X$ be a free action and assume that $X$ is second-countable. Then $G \curvearrowright X$ admits slices.
\end{theorem}
\begin{proof}
  Fix an embedding with closed image $G \subseteq \GL(n, \RR)$ and consider the vector $v = 1 \in \rM_n(\RR) = V$.  By \cref{lem:map-to-vectorspace-equivariant}, there is a neighbourhood $A \subseteq X$ of $x$ and a continuous map $F\colon X \to V$ such that $F(x) = v$ and $F(ky) = kF(y)$ for all $k \in K$ and $y \in A$.  Since $G \subseteq \GL(n, \RR)$ is closed, its action on $V$ is proper, so that \cite[Theorem 2.3.2]{palais1961-slices} shows the existence of a $G$-slice $S_0$ at $v$.  Put $S = F^{-1}(S_0) \cap A$.  Then $x \in S$, since $F(x) = v$.  We have to show that $S$ is a $K$-slice and then that that $x \in \boxint{S}$.
  
  Let $(k_1,x_1), (k_2,x_2) \in K \times S$.  If $k_1x_1 = k_2x_2$, then
  \begin{gather*}
    k_1F(x_1) = F(k_1x_1) = F(k_2x_2) = k_2F(x_2)
    \eqcomma
  \end{gather*}
  implying that $k_1 = k_2$, since $S_0$ is a $G$-slice.  So $x_1 = x_2$ follows, proving injectivity of the map $K \times S \to KS$.  Since it is also continuous and $KS$ is compact, it is a homeomorphism, and we can conclude that $S$ is a $K$-slice.

  Now fix a neighbourhood $B \subseteq A$ of $x$ and a symmetric identity neighbourhood $L$ such that $L^2 \subseteq K$ and $L \cdot B \subseteq A$.  We claim that the neighbourhood $F^{-1}(L^\circ S_0) \cap B$ of $x$ is contained in $KS$, which will prove that $x \in \boxint{S}$.  Take $y \in F^{-1}(L^\circ S_0) \cap B$.  There is $k \in L$ satisfying $kF(y) \in S_0$.  Since $k \in K$ and $y \in A$, we find that $F(ky) = k F(y) \in S_0$, that is $ky \in F^{-1}(S_0)$.  As we also have $ky \in LB \subseteq A$, we infer that $ky \in S$.   Using the fact that $L$ is symmetric, we now conclude that $y \in L^{-1}S \subseteq KS$, which finishes the proof.
\end{proof}

\begin{corollary}
  \label{cor:K-slices-exist-polynomial-growth}
  Let $G \curvearrowright X$ be a free action of a connected Lie group of polynomial growth and assume that $X$ is second-countable. Then $G \curvearrowright X$ admits slices.
\end{corollary}
\begin{proof}
  Let $K$ be an identity neighborhood in $G$ and let $x \in X$. By \cite[Theorem 2]{losert2001} $G$ contains a maximal compact subgroup $H$. The quotient group $G/H$ is then a Lie group of polynomial growth containing no nontrivial compact normal subgroups, hence by \cite[Corollary 3.6]{losert2020} $G/H$ is a matrix Lie group. Denote by $\pi \colon G \to G/H$ and $p \colon X \to X/H$ the quotient maps. Since $G/H$ acts freely on $X/H$, we can apply \cref{thm:K-slices-exist-matrix-groups} to get a $\pi(K)$-slice $S' \subseteq G/H$ such that $[x] \in \boxint{S'}$.  Set $T = p^{-1}(S') \subseteq X$.  Then $T$ is $H$-invariant, so we can consider the action $H \curvearrowright T$.  Since $H$ is compact there exists by \cite[Theorem 2.1]{mostow57} an $H$-slice $S \subseteq T$ such that $x \in \boxint{S}$, where the interior is taken in the $H$-space $T$.

  We claim that $S$ is a $K$-slice. To see this, let $g,g' \in K$ and $x,x' \in S$ satisfy $gx=g'x'$. Then $(gH)[x] = (g'H)[x']$, so since $gH,g'H \in \pi(K)$ and $[x],[x'] \in \pi(S) \subseteq \pi(\pi^{-1}(S')) = S'$, we get that $gH = g'H$ and $[x]=[x']$. Let $h_1,h_2 \in H$ be such that $g' = gh_1$ and $x = h_2x'$. Then $gx = gh_1h_2x$, so by freeness of the action $h_1h_2 = e$. But $x = h_2x'$ also implies that $h_2=e$ since $h_2 \in H$ and $x,x' \in S$, so we arrive also at $h_1 = e$. We conclude that $g=g'$. It remains to argue that $x \in \boxint{S}$. Since $S$ is an $H$-slice at $x$ in $T$, there is a subset $U \subseteq HS$ which is open in $T$ and contains $x$. Let $W = U \cap \boxint{T}$.  Then $\mathring{K}W \subset \mathring{K}\boxint{T}$ is relatively open and the latter set is open in $X$. So $\mathring{K}W$ is an open subset of $X$ which contains $x$.
\end{proof}


\section{Long covers and tube dimension}
\label{sec:box-dimension}

The notion of tube dimension was introduced for $\mathbb{R}$-actions in \cite[Definition 8.6]{hirshbergszabowinterwu2017}. Here we extend the notion to actions of arbitrary locally compact groups. Recall that the \emph{multiplicity} of a cover $\cU$ of $X$ is the least number $d$ such that the intersection of any $d+1$ elements of $\cU$ is empty.

\begin{definition}
\label{def:tube-dim}
The \emph{tube dimension} of an action $G \grpaction{} X$, denoted by $\tubedim(G \grpaction{} X)$, is the least natural number $d$ such that for all compact subsets $K \subseteq G$ and $Y \subseteq X$ there is a family $\cU$ of open sets of $X$ satisfying the following properties:
\begin{enumerate}
    \item for all $x \in Y$ there is $U \in \cU$ such that $K x \subseteq U$,
    \item every $U \in \cU$ is contained in a tube, and
    \item the multiplicity of $\cU$ is at most $d + 1$.
\end{enumerate}
If no such natural number $d$ exists, then we define $\tubedim(G \grpaction{} X) = \infty$.
\end{definition}

In the present section we will show that for certain group actions there are covers with controlled multiplicity and arbitrary length, as specified by the group action. This is the content of \cref{thm:covering}, which generalizes \cite[Theorem 5.2]{kasprowskiruping17} and follows the same proof outline. As a consequence, we will obtain explicit bounds on the tube dimension of these actions. Recall the number $\rmd(G)$ associated with a group of polynomial growth $G$ from Section~\ref{sec:polynomial-growth}.

\begin{theorem}
  \label{thm:covering}
  Let $G \grpaction{} X$ be an action of a locally compact group of polynomial growth on a locally compact, second-countable Hausdorff space. Suppose that $G \curvearrowright X$ admits slices as in \Cref{def:admits_slices}. Then for every compact identity neighbourhood $K \subseteq G$ there exists an open cover $\cU$ of $X$ with the following properties.
  \begin{enumerate}
  \item The cover consists of open tubes.
  \item The multiplicity of the cover is at most $11^{\rmd(G)} \cdot (\dim X + 1)$.
  \item For every $x \in X$ there exists $U \in \cU$ such that $K \cdot x \subseteq U$.
  \end{enumerate}
In particular
\[ \tubedim(G \curvearrowright X) \leq 11^{\rmd(G)} \cdot (\dim X + 1) - 1 . \]
\end{theorem}
\begin{proof}
  First, note that once we have proved the theorem for $K$, it automatically holds for any compact identity neighbourhood $K' \subseteq K$, hence we may enlarge $K$ to a compact symmetric generating set for $G$. Using \cref{prop:cover_by_translates} with $a=10$, we can find $N \in \bN$ such that $K^{10N}$ can be covered by $11^{\rmd(G)}$ translates of $K^{2N}$. Replacing $K$ with the larger set $K^{2N}$, we now have that $K^5$ can be covered by $11^{\rmd(G)}$ translates of $K$. The first step will be to show the following claim, which is an adaption of \cite[Lemma 4.6]{kasprowskiruping17}.

  \begin{claim}
    \label{claim:locally_finite_cover}
    There is a countable collection $\cS$ of $K^5$-slices such that
    \begin{enumerate}
    \item $\{ (KS)^{\circ} \mid S \in \cS \}$ is a locally finite cover of $X$, and
    \item for every pair $S,S' \in \cS$ there exists $h \in G$ such that $\{ g \in K^3 \mid gS \cap S' \neq \emptyset \} \subseteq h \mathring{K}$.
    \end{enumerate}
  \end{claim}

  \begin{proof}[Proof of Claim~\ref{claim:locally_finite_cover}]
    For every $x \in X$ we can by \ref{cor:K-slices-exist-polynomial-growth} find a $K^5$-slice $S_x'$ such that $x \in S_x' \cap (K^5 S_x')^{\circ}$.  Since $K \subseteq K^5$, each $S_x'$ is also a $K$-slice, hence $x \in (K S_x')^{\circ}$ by \cref{lem:box-interior-independence}. Now consider the cover $\{ (K S_x')^{\circ} \mid x \in X \}$ of $X$. Since $G$ is a locally compact, second countable Hausdorff space, it is paracompact and Lindel{\"o}f.  We can therefore find a countable, locally finite, open refinement for the above cover, i.e. an open cover $\cV = \{ V_i \mid i \in \NN \}$ of $X$ such that every $x \in X$ belongs to only finitely many sets from $\cV$ and such that for every $i \in \NN$ there exists $x(i) \in X$ with $V_i \subseteq (K S_{x(i)}')^{\circ}$.  Since $X$ is locally compact and second countable, every compact subset of $X$ intersects only finitely many elements of $\cV$.  We may hence replace each $V_i$ by $\ol{V_i}^{\circ}$ and assume that it is regular open.  By definition of slices, $K S_{x(i)}'$ is homeomorphic to $K \times S_{x(i)}'$, so we can project the compact set $\overline{V_i} \subseteq K S_{x(i)}'$ to the second coordinate. The resulting  compact set $D_i \subseteq S_{x(i)}'$ is then also a $K$-slice satisfying $V_i \subseteq K D_i$.

    We claim that the cover $\{ (K D_i)^{\circ} \mid n \in \NN \}$ of $X$ is also locally finite.  Indeed, let $i \in \NN$ and suppose $x \in (K D_i)^{\circ}$.  Then there exists an open, precompact neighbourhood $U$ of $x$ such that $U \subseteq KD_i$.  Since $x \in KD_i$ we can write $x = ky$ for some $k \in K$ and $y \in D_i$.  By definition of $D_i$ it follows that there exists $x' \in \ol{V_i}$ such that $x' = k'y$ for some $k' \in K$.  Hence $x' = k'k^{-1}x \in K^2U \cap \ol{V_i}$, in particular $K^2U$ intersects $\ol{V_i}$.  Since $K^2U$ is open, it must intersect the regular open set $V_i$.  Since $K^2\ol{U}$ is compact and $\cV$ is a locally finite cover of $X$, this can happen for at only finitely many $i \in \NN$.  Hence there are only finitely many $i \in \NN$ such that $x \in (KD_i)^{\circ}$.

    For every pair $i,j \in \NN$ such that $D_i \cap K^3 D_j \neq \emptyset$, define the function
    \begin{gather*}
      f_{i,j} \colon D_i \cap K^3 D_j \to K^3
    \end{gather*}
    to be the restriction of the projection $K^3 D_j \to K^3$ to the set $D_i \cap K^3 D_j$.  Then $f_{i,j}$ is continuous by definition of slices, hence we can for every $x \in D_i \cap K^3D_j$ pick an open neighbourhood $U_{i,j,x} \subseteq D_i \cap K^3D_j$ of $x$ such that $f_{i,j}(U_{i,j,x}) \subseteq f_{i,j}(x)K^{\circ}$.  Since $X$ and hence $D_i$ is second countable, we can find an open set $U_{i,j,x}'$ in $D_i$ containing $x$ such that $U_{i,j,x'}' \cap K^3D_j = U_{i,j,x}$.  For every $i \in \NN$ the set $J_i = \{ j \in J \mid D_i \cap K^3D_j \neq \emptyset \}$ is finite, so $U_{i,x} = \bigcap_{j \in J_i} U_{i,j,x}'$ is a finite intersection of open sets, hence open.  Since $D_i$ is compact and $D_i = \bigcup_{x \in D_i} U_{i,x}$, we can find $x_1, \ldots, x_{m_i} \in D_i$ such that $D_i = \bigcup_{k=1}^{m_i} U_{i,x_k}$.  Pick open sets $V_{i,k}$ in $D_i$ with $D_i = \bigcup_{k=1}^{m_i} V_{i,k}$ and $\ol{V_{i,k}} \subseteq U_{i,x_k}$.  Let $S_{i,k}$ denote the closure of $V_{i,k}$ and consider the set
    \begin{gather*}
      \cS = \{ S_{i,k} \mid i \in \NN, 1 \leq k \leq m_i \}
      \eqstop
    \end{gather*}
    Then every element of $\cS$ is a closed subset of $D_i$ for some $i \in \NN$, hence a $K$-slice.
    
    We now argue that $\{ (KS)^{\circ} \mid S \in \cS \}$ is a cover of $X$.  To this end it suffices to show that $(KD_i)^{\circ} = \bigcup_{k = 1}^{m_i} (K V_{i,k})^\circ$ for all $i$.  The inclusion from the right to the left is clear.  To see the other inclusion, let $x \in (K D_i)^{\circ}$ and write $x = ky$ for $k \in K$ and $y \in D_i$.  Since $D_i = \bigcup_{k=1}^{m_i} V_{i,k}$ there is $1 \leq k \leq m_i$ such that $y \in V_{i,k}$.  Further, \cref{lem:box_interior} says that $(KD_i)^{\circ} = \mathring{K} \boxint{D_i}$, so that $k \in \mathring{K}$ follows.  So we find that $x \in \mathring{K}V_{i,k}$, which by the definition of a $K$-slice is open in $KD_i$. We conclude by summarising that $x \in \mathring{K}V_{i,k} \cap (KD_i)^{\circ} \subseteq (KV_{i,k})^{\circ}$.

    We also see that $\{ (KS)^{\circ} \mid S \in \cS \}$ is locally finite since each $(K D_i)^{\circ}$ intersects finitely many sets $KS$ for $S \in \cS$ and the cover $\{ (KD_i)^{\circ} \mid i \in \NN \}$ is locally finite.

    For some $i,j \in \NN$ and $1 \leq r \leq m_i$, $1 \leq s \leq m_j$ put $h = f_{i,j}(x_r)$.  We show that 
    \begin{gather*}
      \{ g \in K^3 \mid gS_{i,r} \cap S_{j,s} \neq \emptyset \} \subseteq h \mathring{K}
    \end{gather*}
    If $g$ is a member of the left-hand side, we can find some $x \in S_{i,r} \cap g S_{j,s} \subseteq U_{i,x_k} \cap gU_{j,x_s} \subseteq D_i \cap K^3 D_{j}$, so $j \in J_i$.  Hence $x \in U_{i,j,x_r}'$.  Since $x \in K^3D_j$ as well, we have that $x \in U_{i,j,x_r}$.  This implies that $g = f_{i,j}(x) \in f_{i,j}(x_r)\mathring{K}$.
  \end{proof}
  
  Let now $\cS = (S_i)_{i \in \NN}$ be as in Claim~\ref{claim:locally_finite_cover}.  By the shrinking lemma \cite[Lemma 41.6]{munkres00} we can find a cover $\{ V_i \mid i \in \NN \}$ of $X$ where each $V_i$ is open and $\ol{V_i} \subseteq (KS_i)^{\circ}$.  Then each $\ol{V_i}$ is a closed subset of $KS_i$, so we can project $\ol{V_i}$ onto $S_i$ to get a new $K$-slice $A_i^0 \subseteq S_i$ satisfying $V_i \subseteq K A_i^0$.  So $\{ (K A_i^0)^{\circ} \mid i \in \NN \}$ is a cover of $X$. Also $A_i^0 \subseteq \ol{V_i} \subseteq (KS_i)^{\circ}$, so that $A_i^0 \subseteq \boxint{S_i}$. Consider the following claim, which is an adaption of \cite[Lemma 5.1]{kasprowskiruping17}.

  \begin{claim}
    \label{claim:dimension_reduction}
    Let $k \in \NN$. If $(A_i)_{i \in \NN}$ is a sequence of compact sets where $A_i \subseteq \boxint{S_i}$ and $\dim A_i \leq k$ for all $i \in \NN$, then there exist regular open sets $B_i \subseteq \boxint{S_i}$ in $S_i$ for each $i \in \NN$ such that
    \begin{enumerate}[label=(\alph*)]
    \item \label{it:dim-red:multiplicity}
      the set $\{ \mathring{K}^5 B_i \mid i \in \NN \}$ has multiplicity at most $11^{\rmd(G)}$, and
    \item \label{it:dim-red:dimension}
      for every $i \in \NN$ the compact set
      \begin{gather*}
        A_i \setminus \bigcup_{j \in \NN} \mathring{K}^3B_j
      \end{gather*}
      has dimension at most $k-1$.
    \end{enumerate}
  \end{claim}
  Before proving the claim we show how it can be used to finish the proof of \cref{thm:covering}.  By \cref{prop:dim_properties} we have that $\dim(A_i^0) \leq \dim(X)$ for each $i \in \NN$.  Applying the claim to $(A_i^0)_{i \in \NN}$, we obtain sets $(B_i^0)_{i \in \NN}$ satisfying the conclusion of Claim~\ref{claim:dimension_reduction} with $k = \dim X$. Set
  \begin{gather*}
    A_i^1 = A_i \setminus \bigcup_{j \in \NN} \mathring{K}^3 B_j
    \eqstop
  \end{gather*}
  Then apply Claim~\ref{claim:dimension_reduction} again to the sequence $(A_i^1)_{i \in \NN}$, obtaining a new sequence of sets $(B_i^1)_{i \in \NN}$.  Continuing like this, we obtain sets $(A_i^k)_{i \in \NN}$ and $(B_i^k)_{i \in \NN}$ for every $k \in \NN$. We will now show that
  \begin{gather*}
    \cU = \{ \mathring{K}^5 B_i^k \mid i \in \NN, 0 \leq k \leq \dim X \}
  \end{gather*}
  satisfies the properties of \cref{thm:covering}.  First, note that since $B_i^k$ is a regular open subset of the $K^5$-slice $S_i$ with $B_i^k \subseteq \boxint{S_i}$, it follows from equations \eqref{eq:box_interior} and \eqref{eq:boxint_open} of \cref{lem:box_interior} that $\mathring{K}^5 B_i^k = (K^5\ol{B_i^k})^{\circ}$ is an open tube. Hence the first assertion of \cref{thm:covering} is established.

  We now claim that
  \begin{gather*}
  \label{eq:covering-family}
    \cU' = \{ \mathring{K}^4 B_i^k \mid i \in \NN, 0 \leq k \leq \dim X \}
  \end{gather*}
  is a cover of $X$.  Indeed, let $x \in X$, so that $x \in K A_i^0$ for some $i \in \NN$, say $x = x' y$ with $x' \in K$ and $y \in A_i^0$.  If $y \in \mathring{K}^3 B_j^0$ for some $j \in \NN$ then $x \in K \mathring{K}^3 B_j^0 = \mathring{K}^4 B_j^0$, so $x$ belongs to an element of $\cU'$.  If no such $j$ exists, then by definition we have $y \in A_i^1$.  Continuing like this, if $y \notin \mathring{K}^3 B_j^k$ for any $j \in \NN$ and $0 \leq k \leq \dim X$, we reach the conclusion that $y \in A_i^{\dim X + 1}$.  However $\dim A_i^{\dim X + 1} \leq -1$ so $A_i^{\dim X + 1} = \emptyset$ which is a contradiction. This shows that $\cU'$ is a cover of $X$.

  Now if $x \in X$, say $x \in \mathring{K}^4 B_i^k$ for some $i \in \NN$ and $0 \leq k \leq \dim X$, then $K x \subseteq K (\mathring{K}^4 B_i^k) = \mathring{K}^5 B_i^k$, which shows that $\cU$ satisfies the second assertion of \cref{thm:covering}. Also, by Claim~\ref{claim:dimension_reduction} the multiplicity of $\cU$, being a union of $\dim X + 1$ sets all of multiplicity at most $11^{\rmd(G)}$, cannot exceed $11^{\rmd(G)} \cdot (\dim X + 1)$.  Hence the last assertion of \cref{thm:covering} is also established.
\end{proof}

As we have just shown, it remains to prove Claim~\ref{claim:dimension_reduction}.
\begin{proof}[Proof of Claim~\ref{claim:dimension_reduction}]
  Let $(A_i)_{i \in \NN}$ be as in the statement of the claim.  We will divide the proof into five steps.

  \vspace{0.5em}

  \noindent \textbf{Step 1}.
  In the first step we will construct the sets $(B_i)_{i \in \NN}$.  For this, fix $i \in \NN$.  We apply the \cref{def:small_inductive_dimension} of small inductive dimension to the set $A_i$ viewed as a subset of the ambient space $S_i$ to obtain for every $x \in A_i$ an open neighbourhood $U_x$ of $x$ in $S_i$ which satisfies $\dim(\partial_{S_i}U_x) \leq k-1$.  As explained in \cref{rmk:inductive_dimension_open_regular}, we may assume that the sets $U_x$ are regular open.

  Since $A_i$ is compact we can find a finite subset $F_i\subseteq A_i$ such that $U_i = \bigcup_{x\in F_i}U_x$ contains $A_i$.  Note that $U_i$ is regular open in $S_i$.  Since $X$ is separable and metrizable, every subset is also separable and metrizable, in particular $\partial_{S_i}U_x$ is so for each $x \in F_i$.  Hence we can apply \cref{prop:dim_properties} to obtain
  \begin{equation}
    \label{eq:dimension_red_def}
    \dim(\partial_{S_i}U_i)
    \leq
    \dim \Big( \bigcup_{x \in F_i} \partial_{S_i}U_x \Big)
    \leq
    k-1. 
  \end{equation}
  We define subsets $(I_i)_{i \in \NN}$ of $\NN$ and sets $(B_i)_{i \in \NN}$ in $X$ recursively.  Set $I_1 = \emptyset$ and $B_1 = U_1$.  Further, for $i \geq 2$, set $I_{i} = \{ j \in \NN \mid j < i, K^2 \ol{B_j} \cap U_i \neq \emptyset \}$ and
  \begin{gather*}
    B_i = U_i \setminus \bigcup_{j \in I_i}K^3\ol{B_j}
    \eqstop
  \end{gather*}
  It is then clear from the construction that $B_i$ is an open subset of $S_i$ for each $i \in \NN$.  We will prove that each $B_i$ is regular open by induction.  For $i = 1$ we have that $B_1 = U_1$ which is regular open by construction.  Next, assume that $B_j$ is regular open for all $j < i$. Note first that
  \begin{gather*}
    \ol{B_i}
    \subseteq
    \ol{U_i} \setminus \Big( \bigcup_{j\in I_i} K^3\ol{B_j} \Big)^{\circ}
    \subseteq
    \ol{U_i} \setminus \bigcup_{j\in I_i} (K^3\ol{B_j})^{\circ}
    \eqstop
  \end{gather*}
  Now by the induction assumption combined with equations \eqref{eq:box_interior} and \eqref{eq:boxint_open} of \cref{lem:box_interior} we have that $(K^3\ol{B_j})^{\circ} = \mathring{K}^3\boxint{\ol{B_j}} = \mathring{K}^3 B_j$ for all $j < i$.  Since $S_i$ is a $K^3$-slice, we have $\ol{WB} = \ol{W}\, \ol{B}$ for all subsets $W \subseteq K^3$ and $B \subseteq S_i$.  So we get that $\ol{(K^3\ol{B_j})^{\circ}} = K^3\ol{B_j}$.  Using this and the fact that $U_i$ is regular open, we find
  \begin{gather*}
    \mathring{\ol{B_i}}
    =
    \Big( \ol{U_i}\setminus\bigcup_{j\in I_i} (K^3\ol{B_j} )^{\circ} \Big)^{\circ}
    \subseteq
    \mathring{\ol{U_i}} \setminus \ol{\bigcup_{j\in I_i}  (K^3\ol{B_j})^{\circ}}
    =
    U_i \setminus \bigcup_{j\in I_i} \ol{ (K^3\ol{B_j})^{\circ} }
    =
    U_i \setminus \bigcup_{j\in I_i} \mathring{K}^3B_j
    =
    B_i
    \eqstop
  \end{gather*}
  Since $B_i \subseteq \mathring{\ol{B_i}}$ is obvious, this proves that $B_i$ is regular open, finishing the induction argument.

  \vspace{0.5em}

  \noindent \textbf{Step 2}.
  In this step we show assertion \ref{it:dim-red:multiplicity} of Claim~\ref{claim:dimension_reduction}.  Note that by construction of $B_i$, the elements of the set $\{KB_i\}_{i\in\NN}$ are pairwise disjoint.  Indeed, suppose for a contradiction that $x \in KB_i \cap KB_j$ with $j < i$, say $x = g_iy_i = g_jy_j$ with $g_i,g_j \in K$, $y_i \in B_i$ and $y_j \in B_j$.  Then $y_i = g_i^{-1}x \in U_i \cap K(KB_j)$, hence $j \in I_i$.  It then follows from the definition of $B_i$ that $y_i \notin K^3 \ol{B_j}$ which contradicts $y_j = g_j^{-1}g_iy_i \in K^2B_j$.

  Consider now the set $\{K^5 B_j\}_{j\in\NN}$ from assertion \ref{it:dim-red:multiplicity} of the claim.  By choice of $K$, its power $K^5$ can be covered by $\ell = 11^{\rmd(G)}$ translates of $K$, say $K^5 \subseteq \bigcup_{n=1}^\ell g_n K$ for some $g_1, \dotsc, g_\ell \in G$.  Suppose for a contradiction that there exist indices $i_1 < \cdots < i_{\ell+1}$ such that there is $x \in K^5B_{i_1} \cap \cdots \cap K^5B_{i_{\ell+1}}$.  Then for each $1 \leq r \leq \ell + 1$ there exists $1 \leq n_r \leq \ell$ such that $g_{n_r}^{-1}x \in KB_{i_r}$.  Since $\{ n_1, \dotsc, n_{\ell+1} \} \subseteq \{ 1 , \dotsc, \ell \}$, there exist $1 \leq r,r' \leq \ell + 1$ such that $n_r = n_{r'}$.  This gives $g_{n_r}^{-1}x = g_{n_{r'}}^{-1}x \in KB_{i_r} \cap KB_{i_{r'}}$, which is a contradiction.  Hence the set $\{ K^5 B_j \}_{j \in \NN}$ has multiplicity at most $\ell$.
  
  \vspace{0.5em}

  \noindent \textbf{Step 3}.
  In this step we will show that the sets $S_j\cap (\partial K^3) \ol{B}_k$ are empty when $j \in \NN$ and $k \in I_j$.

  Suppose for a contradiction that there is some $x \in S_j \cap (\partial K^3) \ol{B}_k$, say $x = gy$ for $g \in \partial K^3$ and $y \in \ol{B_k}$.  Since $k \in I_j$, there exists some $x' \in K^2\ol{B_k} \cap U_j \subseteq K^2\ol{B_k}\cap S_j$, say $x' = g'y'$ with $g' \in K^2$ and $y' \in \ol{B_k}$.  We now have that $g,g' \in K^3$, $S_j \cap g S_k \neq \emptyset$ and $S_j \cap g' S_k \neq \emptyset$, so by Claim~\ref{claim:locally_finite_cover} we obtain that $g'^{-1}g \in \mathring{K}$.  Hence
  \begin{gather*}
    g= g'(g'^{-1}g) \in K^2\mathring{K} = \mathring{K}^3
    \eqcomma
  \end{gather*}
  which contradicts $g \in \partial K^3$.

  \vspace{0.5em}

  \noindent \textbf{Step 4}.
  In this step we prove that showing \ref{it:dim-red:dimension} of the present claim can be reduced to showing that $\partial_{S_j} B_j$ has dimension at most $k-1$ for all $j \in \NN$.  Fix $i \in \NN$ and note that
  \begin{gather*}
    \bigcup_{j \in \NN} \mathring{K}^3B_j
    \supseteq
    \bigcup_{j \leq i} \mathring{K}^3B_j
    =
    \mathring{K}^3\Big( U_i \setminus \bigcup_{j \in I_i} K^3 \ol{B_j} \Big) \cup \bigcup_{j < i} \mathring{K}^3B_j
    \supseteq
    U_i \setminus \bigcup_{j \in I_i} K^3 \ol{B_j} \cup \bigcup_{j < i} \mathring{K}^3B_j
    \eqstop
  \end{gather*}
  Hence, taking complements in $A_i$ and using the fact that $A_i \subseteq U_i$, we obtain
  \begin{align*}
    A_i \setminus \bigcup_{j\in\NN}\mathring{K}^3 B_j
    & \subseteq
      \Big( A_i \setminus  \Big( U_i \setminus \bigcup_{j \in I_i} K^3 \ol{B_j} \Big) \Big) \setminus \bigcup_{j \in I_i} \mathring{K}^3B_j \\
    & =
      \Big( A_i \setminus U_i \cup A_i \cap \bigcup_{j<i} K^3 \ol{B_j} \Big) \setminus \bigcup_{j \in I_i} \mathring{K}^3B_j \\
    & =
      A_i \cap \Big( \bigcup_{j \in I_i} K^3 \ol{B_j} \setminus \bigcup_{j \in I_i} \mathring{K}^3 B_j \Big) \\
    & \subseteq
      \bigcup_{j \in I_i} A_i \cap (K^3\ol{B_j} \setminus \mathring{K}^3 B_j)
  \end{align*}
  Thus, in order to show that $A_i \setminus \bigcup_{j \in \NN} \mathring{K}^3 B_j$ has dimension at most $k-1$, it suffices by \cref{prop:dim_properties} to show that the sets
  \begin{gather*}
    A_i \cap (K^3\ol{B_j} \setminus \mathring{K}^3B_j)\eqcomma \qquad j \in I_i\eqcomma
  \end{gather*}
  have dimension at most $k-1$.  By Step 3, the set $A_i \cap (\partial K^3)\ol{B_j}$ is empty.  Therefore, using \eqref{eq:several-boundary-expressions} of \cref{lem:box_interior}, we have that
  \begin{gather*}
    A_i \cap (K^3\ol{B_j} \setminus \mathring{K}^3B_j)
    =
    (A_i \cap (\partial K^3)\ol{B_j} ) \cup (A_i \cap K^3(\partial_{S_j}B_j) )
    =
    A_i \cap K^3(\partial_{S_j}B_j)
    \eqstop
  \end{gather*}
  For each $i,j \in \NN$, we consider the projection $K^3S_j \to S_j$, which maps $K^3(\partial_{S_j}B_j)$ to $\partial_{S_j}B_j$. The restriction of this projection to $A_i \cap K^3(\partial_{S_j}B_j)$ is a homeomorphism onto its image by \cref{lem:slice_restriction_injective}, so it suffices to show that the dimension of $\partial_{S_j}B_j$ is at most $k-1$.

  \vspace{0.5em}

  \noindent \textbf{Step 5}.
  In this step we finish the proof of Claim~\ref{claim:dimension_reduction} by showing that $\partial_{S_i}B_i$ has dimension at most $k-1$ for all $i \in \NN$. We establish this by induction. Consider first the base case $i=1$. Then $\partial_{S_1}B_1 = \partial_{S_1}U_1$ which has dimension at most $k-1$ by inequality \eqref{eq:dimension_red_def}.

  For the induction step, let $i \in \NN$ and assume that $\dim(\partial_{S_j}B_j) \leq k-1$ for all $j < i$.  First we estimate $\partial_{S_i}B_i$ using \Cref{lem:box_interior} to obtain
  \begin{align*}
    \partial_{S_i}B_i
    & \subseteq
      \partial_{S_i}U_i\cup\bigcup_{j \in I_i}\partial_{S_i}(S_i\cap\mathring{K}^3B_i) \\
    & =
      \partial_{S_i}U_i\cup\bigcup_{j \in I_i} S_i \cap\partial(\mathring{K}^3B_j)\\
    & =
      \partial_{S_i}U_i\cup\bigcup_{j \in I_i} (S_i \cap (\partial K^3)\ol{B_j} ) \cup (S_i \cap K^3(\partial_{S_j}B_j))
      \eqstop
  \end{align*}
  From \eqref{eq:dimension_red_def} we know that $\partial_{S_i}U_i$ has dimension at most $k-1$.  Further, by Step 3, the sets $S_i \cap (\partial K^3)\ol{B_j}$ for $j \in I_i$ are empty.  Applying \cref{prop:dim_properties}, we need only show that $S_i \cap K^3(\partial_{S_j}B_j)$.  But for this we can again consider the projection of the tube $K^3(\partial_{S_j}B_j) \to \partial_{S_j}B_j$ which, once restricted to $S_i \cap K^3(\partial_{S_j}B_j)$, becomes a homeomorphism onto its image by \cref{lem:slice_restriction_injective}.  Since the image is a subset of $\partial_{S_j}B_j$ which has dimension at most $k-1$ by the induction hypothesis, we can appeal to \cref{prop:dim_properties} to finish the proof.
\end{proof}


\section{Partitions of unity}
\label{sec:partitions}

Our next aim is to provide a suitable adaption and generalisation of \cite[Proposition 8.23]{hirshbergszabowinterwu2017}.  The notion of Lipschitz functions, which was used in the context of $\mathbb{R}$-actions, is not suitable for general amenable locally compact groups and we resort to the following notion of F{\o}lner functions instead.  As it turns out, this concept is still suitable to characterise finite tube dimension and can be used for nuclear dimension estimates.
\begin{definition}
  \label{def:folner-function}
  Let $G \curvearrowright X$ be an action and $(Y, \mathrm{d})$ a metric space. Let $K \subseteq G$ be a compact identity neighborhood and let $\epsilon > 0$. Then a function $\Phi: X \to Y$ is called $(K, \varepsilon)$-F{\o}lner if for every $g \in K$ and every $x \in X$ we have
  \begin{gather*}
    \mathrm{d}(\Phi(gx), \Phi(x)) < \varepsilon
    \eqstop
  \end{gather*}
  A function $\varphi: X \to \mathbb{C}$ is called $(K, \varepsilon)$-F{\o}lner if it is $(K, \varepsilon)$-F{\o}lner for the Euclidean metric on $\bC$.
\end{definition}

\begin{remark}
  This notion is not to be confused with the F{\o}lner function of a group.
\end{remark}

\begin{lemma}
  \label{lem:creating-folner-functions}
  Let $G$ be an amenable locally compact group and let $G \grpaction{} X$ be a continuous action by homeomorphisms.  Let $\varphi\colon X \to \CC$ be a continuous bounded function, let $K \subseteq G$ be compact and $\veps > 0$.  If $B \subseteq G$ is a $(K, \varepsilon/\|\phi\|_\infty)$-F{\o}lner set, then
  \begin{gather*}
    \psi(x) = \frac{1}{\haar(B)}\int_B \vphi(g^{-1}x) \rmd \haar(g)
  \end{gather*}
  is $(K, \veps)$-F{\o}lner and continuous.
\end{lemma}
\begin{proof}
  It is clear that $\psi$ is continuous, since $\varphi$ is continuous.  We have to show that it is $(K, \varepsilon)$-F{\o}lner.  For $g \in K$ we find that
  \begin{align*}
    \|g\psi - \psi\|_\infty
    & =
      \|g (\frac{1}{\mathrm{m}(B)} \mathbb{1}_B * \varphi) -  \frac{1}{\mathrm{m}(B)} \mathbb{1}_B * \varphi\|_\infty \\
    & =
      \|(g \frac{1}{\mathrm{m}(B)} \mathbb{1}_B) * \varphi -  \frac{1}{\mathrm{m}(B)} \mathbb{1}_B * \varphi\|_\infty \\
    & \leq 
      \|(g \frac{1}{\mathrm{m}(B)} \mathbb{1}_B) -  \frac{1}{\mathrm{m}(B)} \mathbb{1}_B\|_1 * \|\varphi\|_\infty    \\
    & \leq
      \frac{\varepsilon}{\|\varphi\|_\infty} \| \varphi\|_\infty\\
    & =
      \varepsilon
      \eqstop
  \end{align*}
\end{proof}

The next proposition is an adaption of \cite{hirshbergszabowinterwu2017}, which was formulated for $\mathbb{R}$-actions.  In the general setup of amenable groups, it is not possible any longer to show the existence of Lipschitz partitions of unity, however a slight adaption of the proof makes it possible to replace these by suitable F{\o}lner partitions of unity, which suffice for the purpose of proving nuclear dimension estimates.  There is only a conceptual innovation, but no substantial technical change over the proof of \cite[Proposition 8.23]{hirshbergszabowinterwu2017}.

In order to formulate our proposition, we briefly recall combinatorial simplicial complexes and their $\ell^1$-metric realisation.
\begin{itemize}
\item A (combinatorial) simplicial complex is a set $Z$ of finite sets closed under taking subsets.  An $l$-simplex in $Z$ is a element $\sigma \in Z$ with cardinality $|\sigma| = l + 1$ and we denote by $Z^{(l)}$ the set of all $l$-simplices of $Z$. The elements of $Z^{(0)}$ are called the \emph{vertices} of $Z$. The \emph{dimension} of $Z$ is the highest number $d$ such that $Z^{(d-1)} \neq \emptyset$ (if no such number exists, the dimension of $Z$ is infinite).
\item Given a simplicial complex $Z$, we denote by $\mathrm{C} Z = \{\sigma \in Z \sqcup \{\infty\} \mid \sigma \cap Z^{(0)} \in Z\}$ the cone over $Z$.
\item The $\ell^1$-metric realisation of a simplicial complex $Z$ is
  \begin{gather*}
    |Z| = \bigcup_{\sigma \in Z} \{(z_v)_v \in [0,1]^{Z^{(0)}} \mid \sum_{v \in \sigma} z_v = 1 \text{ and } z_v = 0 \text{ for any } v \in Z_0 \setminus \sigma \}
  \end{gather*}
  endowed with the restriction of the $\ell^1$-metric on finitely supported functions on $Z^{(0)}$.
\item For any vertex $v_0\in Z^{(0)}$, the \emph{simplicial star} around $v$ is the set of simplices that contain $v_0$, and the \emph{open star} around $v_0$ is the subset $\{ (z_v)_v \in |Z| : z_{v_0} > 0 \}$ of the geometric realization of $Z$.
\item If $\mathcal{U}$ is a finite open cover of $X$, its \emph{nerve complex} is the simplicial complex
\[ \mathcal{N}(\mathcal{U}) = \{ \mathcal{U}' \subseteq \mathcal{U} : \bigcap_{U \in \mathcal{U}'} U \neq \emptyset \} . \]
The dimension of $\mathcal{N}(\mathcal{U})$ is equal to $\mathrm{mult}(\mathcal{U}) - 1$.
\item If $\mathcal{U}$ is a finite open cover of a closed subset $C$ of $X$, we denote by $\mathcal{U}^+$ the open cover of $X$ given by $\mathcal{U}^+ = \mathcal{U} \cup \{ X \setminus C \}$. One can then identify $\mathcal{N}(\mathcal{U}^+)$ as a subset of $\mathrm{C}\mathcal{N}(\mathcal{U})$ via sending $X \setminus C$ to $\infty$.
\end{itemize}

We are now ready to formulate our adaption and generalisation of \cite[Proposition 8.23]{hirshbergszabowinterwu2017}.
\begin{proposition}
  \label{prop:characterisation-tube-dimension}
  Let $G$ be an amenable locally compact second countable group and let $G \curvearrowright X$ be a continuous action by homeomorphisms.  Then the following statements are equivalent.
  \begin{enumerate}
  \item \label{it:tube-dim}
    The tube dimension of $G \curvearrowright X$ is at most $d$.

  \item \label{it:map-into-cone}
    For every compact subset $K \subseteq G$, every compact subset $A \subseteq X$ and every $\varepsilon > 0$ there is a finite simplicial complex $Z$ of dimension at most $d$ and a continous map $\Phi: X \to |\mathrm{C} Z|$ such that
    \begin{enumerate}
    \item $\Phi$ is $(K, \varepsilon)$-F{\o}lner,
    \item for every vertex $v \in Z_0$ the preimage of the open star around $v$ under $\Phi$ is contained in a tube, and
    \item $\Phi(A) \subseteq |Z|$.
    \end{enumerate}
    
  \item \label{it:partition-plain}
    For every compact subset $K \subseteq G$, every $\varepsilon > 0$ and every compact subset $A \subseteq X$ there is a finite partition of unity $(\varphi_i)_{i \in I}$ for $A \subseteq X$ such that
    \begin{enumerate}
    \item $\varphi_i$ is $(K,\varepsilon)$-F{\o}lner for all $i$,
    \item $\varphi_i$ is supported in the interior of a tube, and
    \item there is a partition $I = I^{(0)} \sqcup \hdots \sqcup I^{(d)}$ such that for all $l \in \{0, \dotsc, d\}$ and all $i,j \in I^{(l)}$ we have
      \begin{gather*}
        \supp \varphi_i \cap \supp \varphi_j = \emptyset
        \eqstop
      \end{gather*}
    \end{enumerate}

  \item \label{it:partition-fattened}
    For every pair of compact subsets $K, L \subseteq G$, every $\varepsilon > 0$ and every compact subset $A \subseteq X$ there is a finite partition of unity $(\varphi_i)_{i \in I}$ for $A \subseteq X$ such that
    \begin{enumerate}
    \item $\varphi_i$ is $(K,\varepsilon)$-F{\o}lner for all $i$,
    \item $L\supp \varphi_i$ is contained in the interior of a tube, and
    \item there is a partition $I = I^{(0)} \sqcup \hdots \sqcup I^{(d)}$ such that for all $l \in \{0, \dotsc, d\}$ and all $i,j \in I^{(l)}$ we have
      \begin{gather*}
        L\supp \varphi_i \cap L \supp \varphi_j = \emptyset
        \eqstop
      \end{gather*}
    \end{enumerate}

  \item \label{it:cover-fattened}
    For every compact subset $L \subseteq G$ and every compact subset $A \subseteq X$ there is a finite collection $\mathcal{U}$ of open subsets of $X$ that cover $A$ such that
    \begin{enumerate}
    \item each $U \in \mathcal{U}$ is contained in a tube of shape larger than $L$, and
    \item there is a partition $\mathcal{U} = \mathcal{U}^{(0)} \sqcup \dotsm \sqcup \mathcal{U}^{(d)}$ such that for any $l \in \{0, \dotsc, d\}$ and any distinct elements $U_1, U_2 \in \mathcal{U}^{(l)}$ we have
      \begin{gather*}
        L\overline{U_1} \cap L\overline{U_2} = \emptyset
        \eqstop
      \end{gather*}
    \end{enumerate}
  \end{enumerate}
\end{proposition}
\begin{proof}
  We start by proving \ref{it:tube-dim} $\Rightarrow$ \ref{it:map-into-cone}.  Fix the notation of \ref{it:map-into-cone}, let $0 < C$ be such that $3(d+1)C < \varepsilon$ and let $L \subseteq G$ be an $(K,C)$-F{\o}lner set.  Consider $B = LA$.  Then by the definition of tube dimension there is a family of open subsets $\mathcal{U}$ of $X$, each contained in a tube, having multiplicity at most $d+1$ such that for every $x \in LB$ there is $U \in \mathcal{U}$ such that $Lx \subseteq U$. Consider the sets
  \begin{gather*}
    V_U = \{x \in X \mid Lx \subseteq U\} \quad U \in \mathcal{U}
    \text{,}
  \end{gather*}
  and the collection $\mathcal{V} = \{V_U \mid U \in \mathcal{U}\}$. Then $\mathcal{V}$ covers $B$ and we can find a finite subset $\mathcal{U}_0 \subseteq \mathcal{U}$ such that $\mathcal{V}_0 = \{V_U \mid U \in \mathcal{U}_0\}$ remains a cover of $B$.  Let $(\varphi_U)_{U \in \, \mathcal{U}_0}$ be a partition of unity for $B \subseteq X$ subordinate to $\mathcal{V}_0$.  By \cref{lem:creating-folner-functions}, the functions $\psi_U = \frac{1}{\mathrm{m}(L)} \mathbb{1}_L * \varphi_U$ for $U \in \mathcal{U}_0$ are $(K, C)$-F{\o}lner. Note also that $\supp(\psi_U) \subseteq \supp(\mathbb{1}_L)\supp(\phi_U) \subseteq LV_U \subseteq U$ for each $U \in \mathcal{U}_0$. Since $\mathcal{U}_0$ has multiplicity at most $(d + 1)$, at most $2(d+1)$ functions from the set $\{g\psi_U \mid U \in \mathcal{U}_0\} \cup \{\psi_U \mid U \in \mathcal{U}_0\}$ can be non-zero at a time.  So for $g \in K$, we find that
  \begin{gather*}
    \|g(\mathbb{1}_X - \sum_{U \in \mathcal{U}_0} \psi_U) - (\mathbb{1}_X - \sum_{U \in \mathcal{U}_0} \psi_U)\|_\infty
    =
    \|\sum_{U \in \mathcal{U}_0} g\psi_U - \sum_{U \in \mathcal{U}_0} \psi_U \|_\infty
    \leq
    2(d + 1)C
    \text{,}
  \end{gather*} 
which shows that $\mathbb{1}_X - \sum_{U \in \mathcal{U}_0} \psi_U$ is $(K, 2(d+1)C)$-F{\o}lner.

Let $Z = \mathcal{N}(\mathcal{U}_0)$ be the nerve of $\mathcal{U}_0$, which is a simplicial complex of dimension at most $\mathrm{mult}(\mathcal{U}_0) = (d+1)-1 = d$.  Consider the continuous map
\begin{gather*}
  \Phi\colon X \to |\cN(\cU_0^+)| \subseteq |\rC Z|
\end{gather*}
given by $\Phi = (\mathbb{1}_X - \sum_{U \in \, \cU_0} \psi_U) \oplus \bigoplus_{U \in \, \cU_0} \psi_U$ and observe that $\Phi$ is $(K, 3(d+1)C)$-F{\o}lner.  By the choice of $C$, this shows that $\Phi$ is $(K, \varepsilon)$-F{\o}lner.  It maps $A$ into $|\mathcal{N}(\mathcal{U}_0)|$, since $\mathbb{1}_X - \sum_{U \in \mathcal{U}_0} \psi_U$ vanishes on $A$.  Also the preimage of every open star is contained in some $U \in \mathcal{U}_0$, which in turn is contained in a tube.

Next we prove \ref{it:map-into-cone} $\Rightarrow$ \ref{it:partition-plain}. Let $K \subseteq G$ be a compact subset, let $A \subseteq X$ be compact and let $\varepsilon > 0$.  Let $\Phi: X \to |CZ|$ be chosen as \ref{it:map-into-cone} for $K \subseteq G$ and the constant $\frac{\varepsilon}{2 (d+1)(d+2)(2d+3)}$.  Let $I^{(l)}$ be the collection of $l$-simplices of $Z$ for $l \in \{0,\dotsc, d\}$, and let $I = I^{(0)} \cup \cdots \cup I^{(d)}$, which is a disjoint union. Consider the functions $\nu_\sigma: |\mathrm{C} Z| \to [0,1]$ for $\sigma \in Z$ as in \cite[Lemma 8.18]{hirshbergszabowinterwu2017}: They are $2(d+1)(d+2)(2d+3)$-Lipschitz and form a partition of unity for $|Z| \subseteq |\mathrm{C} Z|$ which is subordinate to the cover $\{ V_\sigma : \sigma \in Z \}$ where
\begin{gather*}
  V_\sigma
  = \left \{ (z_v)_v \in |Z| \mid z_v > z_{v'} \text{ for all } v \in \sigma \text{ and } v' \in Z_0\setminus \sigma \right \}
  \eqstop
\end{gather*}
Put
\begin{gather*}
  \varphi_\sigma = \nu_\sigma \circ \Phi: X \to [0,1]
  \eqstop
\end{gather*}
We get that $\sum_{\sigma \in I} \varphi_\sigma(x) = 1$ for $x \in A$ since $\Phi(A) \subseteq |Z|$ and $\sum_{\sigma} \nu_\sigma (z) = 1$ for $z \in |Z|$.  Since $\Phi$ is $(K, \frac{\varepsilon}{2 (d+1)(d+2)(2d+3)})$-F{\o}lner and $\nu_\sigma$ is $2(d+1)(d+2)(2d+3)$-Lipschitz, we find that for $g \in K$ we have
\begin{align*}
  |\nu_\sigma \circ \Phi(g x) - \nu_\sigma \circ \Phi(x)|
  & \leq
    2(d+1)(d+2)(2d+3) \mathrm{d}(\Phi(gx), \Phi(x)) \\
  & \leq
    2(d+1)(d+2)(2d+3) \frac{\varepsilon}{2(d+1)(d+2)(2d+3)} \\
  & =
    \varepsilon
    \text{,}
\end{align*}
so that $\varphi_\sigma$ is an $(K, \varepsilon)$-F{\o}lner function.  The remaining properties of $(\varphi_\sigma)_\sigma$ follow as in \cite{hirshbergszabowinterwu2017}: since the support of $\nu_\sigma$ is contained in an open star, the support of $\varphi_\sigma$ is contained in the interior of a tube, and the orthogonality condition on $\nu_\sigma$ implies the one on $\varphi_\sigma$.

We proceed to prove \ref{it:partition-plain} $\Rightarrow$ \ref{it:partition-fattened}. Let $K, L \subseteq G$ be compact, let $\varepsilon > 0$ and $A \subseteq X$ be compact.  Without loss of generality, we may enlarge $K$ and assume that $L \subseteq K$.

Choose constants $\delta_1, \delta_2 > 0$ such that
\begin{gather*}
  \delta_2 < \delta_1\text{,} \quad
  \frac{2(d+1)\delta_1}{(1 - (d+1)\delta_1)^2} < \frac{\varepsilon}{2} \quad \text{ and } \quad
  \frac{\delta_1 + \delta_2}{1 - (d+1)\delta_1} < \frac{\varepsilon}{2}
  \eqstop
\end{gather*}
By \ref{it:partition-plain} there is a partition of unity $(\psi_i)_{i \in I}$ for $A \subseteq X$ that consist of $(K, \delta_2)$-F{\o}lner functions, which are supported in the interior of a tube and satisfy the disjointness condition on their supports as in \ref{it:partition-plain}.  Put $\psi'_i = (\psi_i - \delta_1)_+: X \to [0, 1 - \delta_1]$.  Then $\psi'_i$ is $(K, \delta_1 + \delta_2)$-F{\o}lner. Further, the F{\o}lner condition for $\psi_i$ and the fact that $L \subseteq K$ imply that
\begin{gather*}
  L\supp \psi'_i
  \subseteq
  \{x \in X \mid \max_{g \in L} \psi_i(g^{-1}x) \geq \delta_1\}
  \subseteq
  \{x \in X \mid \psi_i(x) \geq \delta_1 - \delta_2\}
  \subseteq
  \supp \psi_i
  \eqstop
\end{gather*}
So $L \supp \psi'_i$ is contained in the interior of a tube.

Put 
\begin{gather*}
  \psi' = \left (\sum_{i \in I} \psi'_i \right ) + \left (\mathbb{1}_X - \sum_{i \in I} \psi_i \right )
  \eqstop
\end{gather*}
Then $\mathbb{1}_X - \psi' = \sum_{i \in I} \psi_i - \psi'_i$ and the image of the latter is contained in the interval $[0, (d+1)\delta_1]$, owing to the fact that $(\psi_i)_{i \in I}$ is the union of $d+1$ orthogonal families of functions. We note for later use that $\psi'$ is bounded from below by $1 - (d+1)\delta_1$. Put
\begin{gather*}
  \varphi_i = \frac{\psi'_i}{\psi'}
  \eqstop
\end{gather*}
Then $(\varphi_i)_{i \in I}$ is a partition of unity for $A \subseteq X$, since $(\sum_i \psi_i)(x) = 1$ for all $x\in A$. Further, $L\supp \varphi_i = L\supp \psi'_i$ is contained in the interior of a tube.  In order see that each $\varphi_i$ is a $(K, \varepsilon)$-F{\o}lner function, we calculate for $g \in K$ and $x \in A$ that
\begin{align*}
  |\varphi_i(gx) - \varphi_i(x)|
  & =
    \left |\frac{\psi'_i(gx)}{\psi'(gx)} - \frac{\psi_i(x)}{\psi'(x)} \right | \\
  & \leq
    \left |\frac{\psi'_i(gx) - \psi'_i(x)}{\psi'(gx)} \right | + \left | \frac{\psi'_i(x)}{\psi'(gx)} - \frac{\psi'_i(x)}{\psi'(x)} \right | \\
  & \leq
    \left |\frac{\psi'_i(gx) - \psi'_i(x)}{\psi'(gx)} \right | + |\psi'_i(x)| \left | \frac{\psi'(x) - \psi'(gx)}{\psi'(gx) \psi'(x)} \right | \\
  & \leq
    \frac{\delta_1 + \delta_2}{1 - (d+1)\delta_1} + 1 \frac{2(d+1)\delta_1}{(1 - (d+1)\delta_1)^2} \\
  & <
    \varepsilon
    \eqstop
\end{align*}

Let us next prove \ref{it:partition-fattened} $\Rightarrow$ \ref{it:cover-fattened}.  Fix $L \subseteq G$ be compact and $A \subseteq X$ be compact as given by \ref{it:cover-fattened}, let $K \subseteq G$ be some compact subset and let $(\varphi_i)_{i \in I}$ be a partition of unity for $A \subseteq X$ as provided by \ref{it:partition-fattened}.  Then $U_i = (\supp \varphi_i)^\circ$, $i \in I$ defines the desired open cover of $A$.

The implication \ref{it:cover-fattened} $\Rightarrow$ \ref{it:tube-dim} is immediate from the definition of the tube dimension.
\end{proof}


\section{Nuclear dimension estimates from tube dimension}
\label{sec:nuclear-dimension}

In this section we obtain upper bounds for the nuclear dimension of crossed products by connected groups in terms of the tube dimension of a dynamical system.  We proceed by first establishing results on transformation groupoids, generalising results from \cite[Section 7]{hirshbergwu2021} and then employ these in a technically refined approach compared to the similar proof strategy of \cite[Theorem 8.1]{hirshbergwu2021}.

\subsection{Restriction to open tubes}
\label{subsec:restriction-tube}

We start by describing how restrictions to open tubes affect transformation groupoids and crossed product \Cstar-algebras.  For the present subsection, we fix the following notation.  Let $G \grpaction{} X$ be an action and set $\cG = G \ltimes X$. Let $K$ be a compact identity neighbourhood of $G$ and let $S$ be a $K$-slice in $X$. We are interested in the relationship between the restriction $\mathcal{G}|_{KS}$ to the corresponding tube $KS$ and the product groupoid $\pairtube = \pair(K) \times S$, where $\pair(K)$ is the pair groupoid associated with $K$ and $S$ is considered as a topological space. Note that by definition of a $K$-slice these groupoids have homeomorphic unit spaces since $(\mathcal{G}|_{KS})^{(0)} = KS$ while $(\pairtube)^{(0)} = K \times S$. We also set $\pairtubeopen = \pair(K^{\circ}) \times \boxint{S}$, which is an open subgroupoid of $\pairtube$.

\begin{proposition}
  \label{prop:pair-groupoid}
  Consider the map
  \begin{gather*}
    \tau \colon
    \pairtube \to \mathcal{G}|_{KS} \colon
    (g,h,x) \mapsto (gh^{-1},hx)
    \eqstop
  \end{gather*}
  Then the following statements hold true.
  \begin{enumerate}
  \item \label{it:pair-groupoid:cont-inj}
    $\tau$ is a continuous, injective groupoid homomorphism.
  \item \label{it:pair-groupoid:image}
    The image of $\tau$ is equal to
    \begin{gather*}
      \{ (g,x) \in \cG|_{KS} : \mathrm{proj}(x) = \mathrm{proj}(gx) \}
      \eqcomma
    \end{gather*}
    where $\mathrm{proj} \colon KS \to S$ denotes the tube projection.  Similarly, we have
    \begin{gather*}
      \tau(\pairtubeopen) = \{ (g,x) \in \cG|_{(KS)^\circ} : \mathrm{proj}(x) = \mathrm{proj}(gx) \}
    \end{gather*}
  \item \label{it:pair-groupoid:clopen-image}
    The set $\tau(\pairtubeopen)$ is clopen in $\cG|_{(KS)^\circ}$.
  \end{enumerate}
\end{proposition}
\begin{proof}
  First observe that $\tau$ is well-defined, as for $g,h \in K$ and $x \in S$ the computation
  \begin{align*}
    hx & = s(gh^{-1},hx) \eqcomma \\
    gx & = gh^{-1}hx = r(gh^{-1},hx)
    \eqcomma
  \end{align*}
  shows that $(gh^{-1},hx) \in \mathcal{G}|_{KS}$.

  We start by showing \ref{it:pair-groupoid:cont-inj}.  In order to show that $\tau$ is a groupoid homomorphism, we compute for $(g,h,x),(h,k,x) \in \pairtube$ that
  \begin{gather*}
    \tau((g,h,x),(h,k,x))
    =
    \tau(g,k,x)
    =
    (gk^{-1},kx)
    =
    (gh^{-1},hx)(hk^{-1},kx)
    =
    \tau(g,h,x)\tau(h,k,x)
    \eqstop
\end{gather*}
  Its defining formula shows that $\tau$ is continuous. To show injectivity, suppose that for a pair of elements $(g,h,x),(g',h',x') \in \pairtube$ we have $\tau(g,h,x) = \tau(g',h',x')$.  In different terms, we suppose $(gh^{-1}, hx) = g'(h')^{-1}, h'x')$.  Then $hx=h'x'$ implies $h=h'$ and $x=x'$ by the definition of a slice, so that subsequently $gh^{-1} = g'h'^{-1} = g'h^{-1}$ forces $g=g'$. This shows that $\tau$ is injective.

  We next prove \ref{it:pair-groupoid:image}.  If $(g,h,x) \in \pair(K) \times S$, its image under $\tau$ belongs to the claimed set, since $\mathrm{proj}(hx) = x = \mathrm{proj}(gx) = \mathrm{proj}(gh^{-1}hx)$.  Conversely, given an element $(g',x') \in \cG|_{KS}$ that satisfies $\mathrm{proj}(x') = \mathrm{proj}(g'x')$, we write $x' = hx$ and $g'x' = gx$ for certain $h,g \in K$ and $x \in S$. Then it follows that $g'hx = gx$ so that freeness of the action $G \grpaction{} X$ implies that $g' = gh^{-1}$. Hence we obtain that $(g',x') = (gh^{-1},hx) = \tau(g,h,x)$ lies in the image of $\tau$.  Using equation \eqref{eq:box_interior} of \Cref{lem:box_interior}, saying that $(KS)^\circ = K^{\circ}\boxint{S}$, we infer that also $\tau(\pairtubeopen) = \{ (g,x) \in \cG|_{(KS)^\circ} : \mathrm{proj}(x) = \mathrm{proj}(gx) \}$ holds.

  Let us finish by proving \ref{it:pair-groupoid:clopen-image}.  Since $\mathrm{proj}$ is continuous, the description of $\tau(\pairtubeopen)$ obtained in \ref{it:pair-groupoid:image} shows that $\tau(\pairtubeopen)$ is closed in $\cG|_{(KS)^{\circ}}$. Further, since $\tau$ is a continous injection from a compact space to a Hausdorff space, it is a homeomorphism onto its image. In particular $\tau(\pairtubeopen)$ is open in $\cG|_{KS}$ since $\pairtubeopen$ is open in $\pairtube$. To conclude, we use the fact that $\cG|_{(KS)^{\circ}}$ is an open subset of $\cG|_{KS}$, so that it follows that $\tau(\pairtubeopen)$ is open in $\cG|_{(KS)^{\circ}}$.
\end{proof}
Combining the previous result with \cite[Lemma 6.1 and Theorem 6.2]{hirshbergwu2021}, we obtain the following maps.
\begin{corollary}
  \label{cor:conditional-expectation}
  Consider the inclusion $\pairtubeopen \subseteq \cG|_{(KS)^{\circ}}$ given by the map $\tau$ of \cref{prop:pair-groupoid}. The following statements hold true.
  \begin{enumerate}
  \item The inclusion $\contc(\pairtubeopen) \to \contc(\cG|_{(KS)^{\circ}})$ extends to an injective *-homomorphism $\Cstarred(\cH) \to \Cstarred(\cG|_{(KS)^{\circ}})$.
  \item The restriction $\contc(\cG|_{(KS)^{\circ}}) \to \contc(\pairtubeopen)$ extends to a conditional expectation $\rE \colon \Cstarred(\cG|_{(KS)^{\circ}}) \to \Cstarred(\pairtubeopen)$.
  \end{enumerate}
\end{corollary}
The next lemma strongly localises subsets of tubes in the presence of connectedness.  Its special case for $\RR$ was used without explicit reference in \cite[p.36, line after (8.1.4)]{hirshbergwu2021}.
\begin{lemma}
  \label{lem:concentration-in-slices}
  Let $G \grpaction{} X$ be an action.  Let $L \subseteq G$ be a connected open subset and let $S$ be a $K$-slice.  If $Lx \subseteq (KS)^\circ$ for some $x \in X$, then there is $s \in S$ such that $Lx \subset Ks$.
\end{lemma}
\begin{proof}
  Write $B = KS$, denote by $\pi \colon B \to S$ the projection on the slice $S$ and consider the map $\psi \colon L \to S$ given by $l \mapsto \pi(lx)$.  We have to show that the image of $\psi$ is a singleton.  Since $\psi$ is continuous and $L$ is connected its image $\psi(L)$ is connected too.  Assume that $s \in \psi(L)$.  Since $Lx \subseteq B^\circ$, there is $g \in L$ and $k \in K^\circ$ such that $gx = ks$. Let $U$ be a symmetric identity neighbourhood in $G$ such that $Uk \subseteq K$ and $Ug \subseteq L$.  Then $ugx = uks \in Ks$ for all $u \in U$, so that $\psi^{-1}(\{s\})$ has non-empty interior.  So $L = \bigsqcup_{s \in S} \psi^{-1}(\{s\})$ is a partition of $L$ into sets with non-empty interior.  Since $G$ is second countable, this implies that the image of $\psi$ is countable.  So $\psi(L)$ is a countable, connected Hausdorff space, and hence a singleton.
\end{proof}

\subsection{The estimate}
\label{sec:estimate}

The aim of this section is to establish bounds on the nuclear dimension of crossed products associated with actions connected amenable groups in terms of the tube dimension of the action.  We will apply this to Lie groups of polynomial growth, for which estimates on the tube dimension can be established.

The next results shows that F{\o}lner functions, as introduced in \cref{sec:partitions}, give rise to quasi-central elements in the $\Lone$-crossed products.  This replaces arguments based on Lipschitz functions used in \cite{hirshbergszabowinterwu2017,hirshbergwu2021} that are only adapted to $\RR$-flows.
\begin{lemma}
  \label{lem:folner-quasi-central-for-I-norm}
  Let $G \curvearrowright X$ be an action and consider the groupoid $G \ltimes X$. Let $K \subseteq G$ be a symmetric compact subset and let $A \subseteq X$ be a compact subset.  Assume that $(\vphi_i)_{i \in I}$ are $(K, \veps)$-F{\o}lner functions in $\contb(X)$ such that $(\vphi_i^2)_i$ is a partition of unity for $A$.  Further assume that $I = I^{(0)} \sqcup \dotsm \sqcup I^{(d)}$ such that $\supp \vphi_i \cap \supp \vphi_j = \emptyset$ for all $0 \leq l \leq d$ and $i,j \in I^{(l)}$.  Then for all $a \in \contc(K \times A) \subseteq \contc(G \ltimes X)$, we have
  \begin{gather*}
    \|\sum_i \vphi_i a \vphi_i - a\|_I \leq 2 (d + 1) \veps^2 \|a\|_I
    \eqstop
  \end{gather*}
\end{lemma}
\begin{proof}
  Let $a \in \contc(G \times X)$ and $f \in \contb(X)$.  Then we have
  \begin{gather*}
    (af)(g,x) = a(g,x) f(x) \quad \text{ and } \quad (fa)(g,x) = a(g,x)f(gx)
    \eqstop
  \end{gather*}
  So we calculate for $a \in \contc(K \times A)$, for $g \in G$ and $x \in A$ that
  \begin{align*}
    |\sum_i \vphi_i(x)\vphi_i(gx) - 1|
    & =
      |\sum_i \vphi_i(x)\vphi_i(gx) - \vphi_i^2(x)| \\
    & =
      |\sum_i \vphi_i(x)(\vphi_i(gx) - \vphi_i(x))| \\
    & \leq
      \sum_i \vphi_i^2(x) \cdot \sum_i (\vphi_i(gx) - \vphi_i(x))^2 \\
    & \leq
      1 \cdot 2 (d + 1) \veps^2
      \eqstop
  \end{align*}
  For the last inequality we used the fact that $\vphi_i$ is a $(K, \veps)$-F{\o}lner function for each $i$ as well as the fact that there are at most $d + 1$ indices $i$ satisfying $\vphi_i(x) \neq 0$ and likewise there are at most $d + 1$ indices $i$ satisfying $\vphi_i(gx) \neq 0$. We can now estimate the following integrals for fixed $x$:
  \begin{align*}
    \int_G | \sum_i \vphi_i a \vphi_i -  a|(g,x) \rmd g
    & =
      \int_G |a(g,x)|  \cdot |\sum_i \vphi_i(gx) \vphi_i(x) -  1| \rmd g \\
    & \leq
      2 (d + 1) \veps^2 \int_G |a(g,x)|   \rmd g \\
    & \leq
      2 (d + 1) \veps^2 \|a\|_I
  \end{align*}
  and similarly
  \begin{align*}
    \int_G | \sum_i \vphi_i a \vphi_i -  a|(g^{-1},gx) \rmd g
    & =
      \int_G |a(g^{-1},gx)| \cdot  |\sum_i \vphi_i(x) \vphi_i(g^{-1}x) -  1| \rmd g \\
    & \leq
      2 (d + 1) \veps^2 \|a\|_I
    \eqstop
  \end{align*}
Taking the supremum of both terms over $x \in X$, we conclude the proof.
\end{proof}
We are now ready to state and prove our main result, establishing nuclear dimension bounds for free actions in terms of the tube dimension and the dimension of the base space.
\begin{theorem}
  \label{thm:finite-nuclear-dimension}
  Let $G \grpaction{} X$ be a free action of a connected amenable group on a second-countable, locally compact Hausdorff space. Then
  \begin{gather*}
    \dimnuc(\conto(X) \rtimes G) 
    \leq
    (\dim X + 1) (\tubedim(G \grpaction{} X) + 1) - 1
  \end{gather*}
\end{theorem}
\begin{proof}
  Set $\mathcal{G} = G \ltimes X$. Let $B_r \subseteq \contc(G,\contc(X))$ be the $r$-Ball with respect to the $I$-norm. Then $\bigcup_{r \in \NN} \ol{B_r}^{\| \cdot \|} = \conto(X) \rtimes G$, so that a Baire category argument shows that there is some $R > 0$ such that the closure of $B_R$ contains the unit ball of $\conto(X) \rtimes G$. Let $F$ be a finite subset of $B_R$ and let $\veps > 0$. Let $L \subseteq G$ and $A \subseteq X$ be compact subsets such that $F \subseteq \contc(L, \contc(A))$.  Write $d = \tubedim(G \grpaction{} X)$ and pick $\delta < ( \frac{\veps}{2(d+1)R} )^{1/2}$.  Then \cref{prop:characterisation-tube-dimension} provides a finite family of functions $(\vphi_i)_{i \in I}$ such that
  \begin{itemize}
  \item $(\vphi_i^2)_{i \in I}$ is a partition of unity for $A \subseteq X$
  \item each function $\vphi_i$ is $(L, \delta)$-F{\o}lner
  \item $L \supp \vphi_i$ is contained in a the interior of a tube $K_i S_i$, and
  \item there is a partition $I = I^{(0)} \sqcup \hdots \sqcup I^{(d)}$ such that for all $l \in \{0, \dotsc, d\}$ and all $i,j \in I^{(l)}$ we have
    \begin{gather*}
      L\supp \varphi_i \cap L\supp \varphi_j = \emptyset
      \text{.}
    \end{gather*}
  \end{itemize}
  For each $i \in I$ denote by $\pairtubeopeni$ the product groupoid $\pair(K_i^{\circ}) \times \boxint{S_i}$ as in \cref{subsec:restriction-tube}. Let $\Psi_i \colon \Cstar(\pairtubeopeni) \to \Cstar(\cG)$ be the inclusion as in \cref{cor:conditional-expectation} and let $\rE_i \colon \Cstar(\cG|_{(K_iS_i)^{\circ}}) \to \Cstar(\pairtubeopeni)$ be the conditional expectation from \cref{cor:conditional-expectation}.  We observe that compression with $\varphi_i$ is a completely positive map on $\Cstar(\cG)$, which is contractive since $\| \varphi_i \|_\infty \leq 1$ and has its image in $\Cstar(\cG|_{(K_i S_i)^{\circ}})$ since $\supp(\varphi_i) \subseteq (K_iS_i)^{\circ}$.  Thus, we obtain a completely positive contractive map $\Phi_i \colon \Cstar(\cG) \to \Cstar(\pairtubeopeni)$ by putting
  \begin{gather*}
    \Phi_i(a)
    =
    \rE_i(\vphi_i a \vphi_i)
    \eqcomma \quad \text{for all } a \in \Cstar(\cG)
    \eqstop
  \end{gather*}

  For $l \in \{0, \dotsc, d\}$ put $A^{(l)} = \bigoplus_{i \in I^{(l)}} \Cstar(\pairtubeopeni)$ and $A = \bigoplus_{l=0}^d A^{(l)}$. Let $\Phi^{(l)} = \bigoplus_{i \in I^{(l)}} \Phi_i\colon \Cstar(\cG) \to A^{(l)}$ and let $\Psi^{(l)} = \bigoplus_{i \in I^{(l)}} \Psi_i\colon A^{(l)} \to \Cstar(\cG)$. Then each $\Phi^{(l)}$ is completely positive and contractive, so $\Phi = \bigoplus_{l \in \{0, \dotsc, d\}} \Phi^{(l)} \colon \Cstar(\cG) \to A$ is also completely positive and contractive.  Moreover, each $\Psi^{(l)}$ is a *-homomorphism, in particular an order zero contraction, so the number $d^{(l)}$ from \Cref{lem:nuclear_lemma} is zero. It follows that $\Psi = \sum_{l \in \{0, \dotsc, d\}} \Psi^{(l)} \colon A \to \Cstar(\cG)$ is completely positive.

  Since $L \supp \vphi_i$ is contained in the interior of $K_i S_i$ and $F \subseteq \contc(L, \contc(A))$ we find that $\vphi_i a \vphi_i \in \Cstar(\pair(K_i^\circ) \times S_i^\circ)$ when $a \in F$.  Indeed,
  \begin{gather*}
    0 \neq \vphi_i a \vphi_i(g,x)
    = \vphi_i(gx)\vphi_i(x) a(g,x)
  \end{gather*}
  implies that $g \in L$ and $gx, x \in \supp \vphi_i$, which in combination with \cref{lem:concentration-in-slices} shows that there is $s \in S_i$ and $k_1, k_2 \in K_i$ such that such that $x = k_1 s$ and $gx = k_2s$.  By freeness of the $G$-action, this implies $g = k_2k_1^{-1}$.  It follows by \cref{prop:pair-groupoid} that $(g, x) = (k_2k_1^{-1}, k_1 s)$ lies in $\tilde{\mathcal{H}}_i = \pair(K_i^\circ) \times \boxint{S_i}$ when identified with its image in $\mathcal{G}$ under $\tau$.  Combined with \cref{lem:folner-quasi-central-for-I-norm} this shows that for all $a \in F$ we have
  \begin{align*}
    \|\Psi \circ \Phi(a) - a\|
    & =
      \| \sum_{i \in I} \vphi_i a \vphi_i - a\| \\
    & \leq
      \| \sum_{i \in I} \vphi_i a \vphi_i - a\|_I \\
    & \leq
      2(d+1)\delta^2 R
    < \veps
      \eqstop
  \end{align*}

  To finish the proof, we observe that by \Cref{ex:pair_groupoid} and \Cref{ex:top_space_groupoid} the $\Cstar$-algebra $\Cstar(\pairtubeopeni)$ is isomorphic to $\Cstar(\pair(K_i^{\circ}) \times \boxint{S_i}) \cong {\cK \otimes \conto(\boxint{S_i})}$ where $\mathcal{K}$ denotes the compact operators on a separable, infinite-dimensional Hilbert space. This has nuclear dimension at most $\dim (\boxint{S_i}) \leq \dim X$, so that \cref{lem:nuclear_lemma} gives
  \begin{gather*}
    \dimnuc(\conto(X) \rtimes G)
    \leq 
    \sum_{l = 0}^d (\dim(X) + 1)(0 + 1) - 1
    =
    (d + 1) (\dim(X) + 1) - 1
    \eqstop
  \end{gather*}
\end{proof}
The previous theorem was formulated in a general form, and we next specialise to concrete nuclear dimension estimates in the context of actions of connected Lie groups of polynomial growth, thanks to our results on the tube dimension established in \cref{sec:box-dimension}.  Recall from \cref{sec:polynomial-growth} Breuillard's numbers $\rmd(G)$ describing the growth asymptotics of balls in a group of polynomial growth such as a nilpotent Lie group.
\begin{corollary}
  \label{cor:finite-nuclear-dimension}
Let $G \grpaction{} X$ be a free action of a connected Lie group of polynomial growth on a second-countable, locally compact Hausdorff space. Then
\begin{gather*}
  \dimnuc(\conto(X) \rtimes G) + 1 \leq
  (\dim X + 1)^2 11^{\rmd(G)}
\end{gather*}
\end{corollary}
\begin{proof}
  By \cref{thm:covering} combined with \cref{cor:K-slices-exist-polynomial-growth} the tube dimension of $G \grpaction{} X$ is bounded by $11^{\rmd(G)} \cdot (\dim X + 1) - 1$.  We hence obtain the estimate from \cref{thm:finite-nuclear-dimension}.
\end{proof}

\begin{remark}\label{rem:nuclear_dimension_bound_difference}
We remark on the difference between the nuclear dimension bound obtained in \Cref{cor:finite-nuclear-dimension} and \cite[Corollary 10.2]{hirshbergszabowinterwu2017} or \cite[Theorem 8.1]{hirshbergwu2021}. These differences are a result of the following two observations about \cite{kasprowskiruping17}:

Firstly, on p.\ 1214 in \cite{kasprowskiruping17} in the final line of the proof of Lemma 5.1, it is written that if $x \in \Phi_{[-4\alpha,4\alpha]}(gB_i)$ then there is $\beta \in \{ -4\alpha, -2\alpha, 0, 2\alpha, 4\alpha \}$ such that $\Phi_\beta(x) \in \Phi_{[-\alpha,\alpha]}(gB_i)$. While this is true, there is a more optimal choice here: $\beta$ can be chosen from the set $\{ -3\alpha,-\alpha,\alpha,3\alpha \}$ which has one less element. This affects the multiplicity in the statement of Kasprowski-R{\"u}ping's Lemma 5.1 (2) and consequently also Theorem 5.2, which is cited in \cite[Theorem 8.8]{hirshbergszabowinterwu2017} and \cite[Theorem 3.11]{hirshbergwu2021}.

Secondly, in the proof of \cite[Theorem 5.2]{kasprowskiruping17} it is written that by Lemma 5.1 (1) every element in $X_{>\gamma}$ is contained in an open set of the form $\Phi_{(-3\alpha,3\alpha)}(gB_i^k)$ for some $g \in G$, $i \in \NN$, $k=0,\ldots, \mathrm{ind} X$. Upon inspection of this claim it seems however that Lemma 5.1 (1) only guarantees that the sets $\Phi_{(-4,\alpha,4\alpha)}(gB_i^k)$ cover $X_{>\gamma}$. Thus, one will in the end need to control the multiplicity of the collection $\mathcal{B} = \{ \Phi_{(-5\alpha,5\alpha)}(B_i^k) \mid i \in \bN, k=0, 1 \ldots, \operatorname{ind} X \}$ rather than $\mathcal{B} = \{ \Phi_{(-4\alpha,4\alpha)}(B_i^k) \mid i \in \bN, k=0,1\ldots,\operatorname{ind}X \}$. This is exactly the situation we end up with in the present paper in the proof of \Cref{thm:covering} after Claim~\ref{claim:dimension_reduction}, and we account for this with the choice of $a=10$ in the very beginning of the proof of \Cref{thm:covering}.
\end{remark}


\section{Classifiable C*-algebras associated with FLC point sets}
\label{sec:classifiable-point-sets}

In this section we describe examples of classifiable \Cstar-algebras arising from Theorem~\ref{thm:finite-nuclear-dimension}.  In greatest generality, abstract conditions on a point set in a connected Lie group of polynomial growth, that ensure that its \Cstar-algebra introduced in \cite{enstadraum2022} is classifiable.  Every nilpotent Lie group is of polynomial growth. Subsequently specialising these results, our main source of concrete examples arise as approximate lattices constructed from arithmetic lattices in products of nilpotent Lie groups.

Let us start by introducing some notation.  We recall that $G$ denotes a locally compact second countable group and we denote by $\cC(G)$ the Chabauty space of all closed subsets of $G$ described whose topology has a subbasis
\begin{align*}
  \cO_K & = \{A \in \cC(X) \mid A \cap K = \emptyset\} \\
  \cO^U & = \{A \in \cC(X) \mid A \cap U \neq \emptyset\}
          \eqcomma
\end{align*}
where $K \subseteq X$ is compact and $U \subseteq X$ is open.  For a closed subset $\Lambda \subseteq G$, we denote its hull and its discrete hull, respectively, by 
\begin{gather*}
  \Omega(\Lambda)
  =
  \ol{ \{g\Lambda\mid g\in G\} }
  \subseteq
  \cC(G)
  \quad
  \text{and}
  \quad
  \Omega_0(\Lambda) = \{P \in \Omega \mid e \in P\}
  \eqstop
\end{gather*}

Using the discrete hull, a Delone set of finite local complexity (FLC) in $G$ can be defined as a closed subset $\Lambda \subseteq G$ such that
\begin{itemize}
\item for every compact subset $K$ there is a finite set $\cF \subseteq \cC(G)$ such that for every $P \in \Omega_0(\Lambda)$ there is $F \in \cF$ such that $P \cap K = F$, and
\item $\emptyset \notin \Omega(\Lambda)$.
\end{itemize}

In order to apply \cref{cor:finite-nuclear-dimension} to dynamical systems arising from point sets, we need the following result, which is well-known for point sets in $\bR^d$, see e.g.\ \cite[Corollary 2.1]{bellissardherrman00}.

\begin{proposition}
  \label{prop:FLC-implies-tdlc}
  Let $\Lambda \subseteq G$ be an FLC Delone set.  Then $\Omega_0(\Lambda)$ is totally disconnected.
\end{proposition}
\begin{proof}
  Let $U$ be a non-empty open subset of $\Omega_0(\Lambda)$.  We have to find a non-empty compact open subset of $U$.  There are relatively compact open subsets $V_1, \dotsc, V_n \subseteq G$ and a compact subset $K \subseteq G$ such that $\emptyset \neq \cO_K\cap\cO^{V_1}\cap \dotsc \cap\cO^{V_n} \subseteq U$.  Let $C = K \cup \ol{V_1} \cup \dotsm \cup \ol{V_n}$.  Since $\Lambda$ is FLC, the set $\cF=\{Q \cap C \mid Q \in \Omega_0(\Lambda)\}$ is finite.  Since $\cF$ consists of finitely many finite subsets of $G$, there is an open identity neighbourhood $V \subseteq G$ satisfying the following condition: for all $F \in \cF$, we have $Q \cap C = F$ if and only if $Q \cap gV \neq \emptyset$ for all $g \in F$ and $Q \cap (C \setminus \bigcup_{g \in F} gV) = \emptyset$.  Writing
  \begin{gather*}
    U_F = \cO_{C \setminus \bigcup_{g \in F} gV} \cap \bigcap_{g \in F} \cO^{gV}
  \end{gather*}
  this can be restated as $Q \cap C = F$ if and only if $Q \in U_F$.  So the open sets $U_F$ for $F \in \cF$ are pairwise disjoint and cover $\Omega_0(\Lambda)$.  They are hence clopen.  Since $U_F \subseteq \cO_K\cap \cO^{V_1} \cap \dotsc \cap \cO^{V_n} \subseteq U$ for some $F \in \cF$, and since $U_F \neq \emptyset$ for all $F \in \cF$ this finishes the proof.  
\end{proof}
Since total disconnectedness is equivalent to zero-dimensionality, we can now derive dimension bounds for the hull of an FLC point set.
\begin{corollary}
  \label{cor:dimension-punctured-hull}
  Let $\Lambda \subseteq G$ be an FLC Delone set in a Lie group. Then $\dim(\Omega(\Lambda)) \leq \dim G$.
\end{corollary}
\begin{proof}
  Since $\Lambda$ is FLC, it is also uniformly discrete, that is there exists an open identity neighbourhood $U \subseteq G$ such that $|U \cap P| \leq 1$ for all $P \in \Omega(\Lambda)$.  

  Put $V = \{P \in \Omega(\Lambda) \mid P \cap U \neq \emptyset\}$, which is an open subset of $\Omega(\Lambda)$ containing $\Omega_0(\Lambda)$.  The map $\vphi_0 \colon V \to U$ satisfying $\{\vphi_0(P)\} = P \cap U$ for all $P \in V$ is continuous by definition of the Fell topology.  So the map $\vphi \colon V \to U \times \Omega_0(\Lambda)$ defined by $\vphi(P) = (\vphi_0(P), \vphi_0(P)^{-1}P)$ is also continuous.  A continuous inverse for $\vphi$ is obtained by restricting the action map for $G \grpaction{} \Omega(\Lambda)$ to obtain a map $\psi \colon U \times \Omega_0(\Lambda) \to V$ satisfying $\psi(g, P) = gP$.  So we infer that $V \cong U \times \Omega_0(\Lambda)$.
  
  In order to conclude the proof, we calculate
  \begin{gather*}
    \dim (\Omega(\Lambda)) = \dim(V) = \dim(U \times \Omega_0(\Lambda)) = \dim(U) + \dim(\Omega_0(\Lambda)) =  \dim(G)
    \eqcomma
  \end{gather*}
  where we used the equalities $\dim (\Omega(\Lambda)) = \sup_{g \in G} \dim (gV) = \dim V$ as well as $\dim U = \sup_{g \in G} \dim gU = \dim G$ which are justified by \cref{prop:dim_properties} together with the fact that $\dim(\Omega_0(\Lambda)) = 0$, which follows from \cref{prop:FLC-implies-tdlc}.
\end{proof}
In order to formulate the first main result of this section, we recall some terminology.  A point set $\Lambda$ is repetitive if for every compact set $K \subseteq G$ there exists a compact set $L \subseteq G$ such that the following statement holds: for all $h \in G$ and all $P \in \Omega(\Lambda)$ there is $g \in G$ such that $g(K \cap P) \subseteq hL \cap \Lambda$.  It follows from \cite[Propositions 2.2 and 2.4]{beckushartnickpogorzelski2020} that an FLC Delone set $\Lambda$ is repetitive if and only if the action $G \grpaction{} \Omega(\Lambda)$ is minimal.  We further recall that, $\Lambda$ is called aperiodic if the action $G \grpaction{} \Omega(\Lambda)$ is free.
In \cite{enstadraum2022} groupoids associated to point sets were introduced.  In particular, if $\Lambda \subseteq G$ is an FLC Delone set, then 
\begin{gather*}
  \cG(\Lambda) = G \ltimes \Omega(\Lambda)|_{\Omega_0(\Lambda)}
  \eqstop
\end{gather*}
is an {\'e}tale groupoid.  The reduced and universal \Cstar-algebra of the point set $\Lambda$ are then by definition
\begin{gather*}
  \Cstarred(\Lambda) = \Cstarred(\cG(\Lambda))
  \quad
  \text{and}
  \quad
  \Cstar(\Lambda) = \Cstar(\cG(\Lambda))
  \eqstop
\end{gather*}
If $G$ is amenable, e.g. a nilpotent Lie group, these \Cstar-algebras are isomorphic.  We will need the following result to analyse $\Cstar(\Lambda)$.
\begin{lemma}
  \label{lem:stable-isomorphism}
  The \Cstar-algebras $C_0(\Omega(\Lambda)) \rtimes_{\mathrm{red}} G$ and $\Cstarred(\Lambda)$ are stably isomorphic.
\end{lemma}
\begin{proof}
  It follows from \cite[Proposition 3.8]{enstadraum2022} and \cite[Example 2.40]{williams2019-toolkit} that the groupoids $G \ltimes \Omega(\Lambda)$ and $\cG(\Lambda)$ are equivalent.  So \cite[Theorem 4.9]{williams2019-toolkit} says that their \Cstar-algebras are Morita equivalent.  Since both are separable, as $\cG(\Lambda)$ is second countable, it follows that $\Cstar(\Lambda) = \Cstarred(\cG(\Lambda))$ is stably isomorphic with $\cont(\Omega(\Lambda)) \rtimes_{\mathrm{red}} G$. 
\end{proof}

\begin{theorem}
  \label{thm:point-set-nuclear-dimension}
  Let $\Lambda$ be an aperiodic, FLC Delone set in a connected Lie group of polynomial growth. Then $\Cstar(\Lambda)$ has finite nuclear dimension.  If additionally, $\Lambda$ is repetitive then the \Cstar-algebra $\Cstar(\Lambda)$ is classifiable.
\end{theorem}
\begin{proof}
  First assume that $\Lambda$ is an aperiodic, FLC Delone set in a connected, simply connected, nilpotent Lie group $G$.  By \Cref{lem:stable-isomorphism}, the \Cstar-algebras $\Cstar(\Lambda)$ and $\Cstarred(\cG(\Lambda))$ are stably isomorphic.  It hence suffices to show that $\cont(\Omega(\Lambda)) \rtimes G$ has finite nuclear dimension.  This follows from \Cref{thm:finite-nuclear-dimension}, making use of the assumptions of aperiodicity, which ensures that $G \grpaction{} \Omega(\Lambda)$ is free.

  Let us now additionally assume that $\Lambda$ is repetitive.  We already know that $\Cstar(\Lambda)$ is separable and has finite nuclear dimension and we will show that $\Cstar(\Lambda)$ is simple, unital, non-elementary and in the UCT class.  It is unital, since $\cG(\Lambda)$ is {\'e}tale and $\cG(\Lambda)\nought = \Omega_0(\Lambda)$ is compact.  Further, since $\cG(\Lambda)$ is infinite, it follows that $\Cstar(\Lambda)$ is not isomorphic to a matrix algebra, and hence not a (unital) elementary \Cstar-algebra. As $G$ is an amenable group, also $G \ltimes \Omega(\Lambda)$ and hence $\cG(\Lambda)$ are amenable groupoids.  By the work of Tu \cite[Proposition 10.7]{tu99} it follows that $\Cstar(\Lambda)$ is in the UCT class.

  It remains to argue that $\Cstar(\Lambda)$ is simple.  Since $\Lambda$ is repetitive the groupoid $G \ltimes \Omega(\Lambda)$ is minimal.  By \cite[Lemma 2.41]{williams2019-toolkit} the orbit spaces of $G \ltimes \Omega(\Lambda)$ and $\cG(\Lambda)$ are homeomorphic so that also the latter groupoid is minimal.  Since $\Lambda$ is aperiodic, $G \ltimes \Omega(\Lambda)$ is principal and so is its subgroupoid $\cG(\Lambda)$.  Finally, observe that $\Cstar(\cG(\Lambda)) = \Cstarred(\cG(\Lambda))$ by amenability of $\cG(\Lambda)$, so that \cite[Theorem 5.1]{brownclarkfarthingsims2014} applies and proves simplicity of $\Cstar(\Lambda)$. 
\end{proof}
We are next going to provide examples to which the previous theorem applies.  It is necessary to verify among others the condition of aperiodicity.  While for repetitive point sets in abelian groups, \cite[Proposition 5.5]{baakegrimm13} shows that aperiodicity is equivalent to the formally weaker requirement that the point stabiliser $G_\Lambda = \{g \in G \mid g\Lambda = \Lambda\}$ is trivial, this is not known in general. The known proof in the abelian case does not generalise to nilpotent groups.  However, in the specific context of model sets introduced below, aperiodicity can be easily verified as we will show.

A \textit{cut-and-project scheme} is a triple $(G,H,\Gamma)$, where $G$ and $H$ are locally compact groups, and $\Gamma$ is a lattice in $G\times H$ such that the projection $\pi_{G}:G\times H\to G$ is injective when restricted to $\Gamma$ and the image of $\Gamma$ under the projection $\pi_H:G\times H\to H$ is dense in $H$.  In view of results obtained in \cite{bjorklundhartnickpgorzelski2018}, we are interested in \emph{regular model sets}, which are obtained as
\begin{gather*}
  \Lambda = \pi_G(\Lambda\cap (G\times W))
\end{gather*}
for windows $W \subseteq H$ satisfying the following conditions: $W$ is a regular compact subset, its boundary is Haar negligible, it satisfies $\partial W \cap \pi_H(\Gamma) = \emptyset$ and $\mathrm{stab}_H(W)=\{e\}$.  If $\Lambda$ is a (regular) model set, then it has finite local complexity by \cite[Proposition 2.13 (i)]{bjorklundhartnick2018}.  As lattices in nilpotent Lie groups are automatically uniform, it follows that (regular) model sets constructed from cut-and-project schemes of nilpotent Lie groups are Delone sets.  Not every lattice in a product of nilpotent Lie groups however gives rise to a cut-and-project scheme and we need to require slightly more: a lattice in a product of groups $\Gamma \leq G \times H$ is \emph{irreducible} if 
the projections onto $G$ and $H$, respectively, are injective and have dense image when restricted to $\Gamma$.

\begin{corollary}
  \label{cor:irreducible-lattice-classifiable}
  Let $G$ be a connected, simply connected nilpotent Lie group and let $\Lambda \subseteq G$ be a regular model set arising from a cut-and-project scheme $(G, H, \Gamma)$ where $H$ is another nilpotent Lie group and $\Gamma$ is an irreducible lattice.  Then $\Cstar(\Lambda)$ is classifiable.
\end{corollary}
\begin{proof}
  By the previous discussion $\Lambda \subseteq G$ is an FLC Delone set.  So \Cref{thm:point-set-nuclear-dimension} applies if we can show that $\Lambda$ is repetitive and aperiodic.  The former follows form \cite[Proposition 3.3]{bjorklundhartnickpgorzelski2018} combined with \cite{beckushartnickpogorzelski2020}.  We have to show that $\Lambda$ is aperiodic.

  Consider the $G$-space $Y = (G \times H) / \Gamma$ where $G$ acts by multiplication on the left in the first component.  By \cite[Theorem 3.1]{bjorklundhartnickpgorzelski2018} there exists a $G$-equivariant map $\beta \colon \Omega(\Lambda) \to Y$.  Since $G_P \subseteq G_{\beta(P)}$ for every $P \in \Omega(\Lambda)$, it suffices to show that the action of $Y$ is a free $G$-space.  If $g' (g,h)\Gamma = (g,h)\Gamma$ for some $g,g' \in G$ and $h \in H$, then $(g^{-1}g'g, e)\Gamma = \Gamma$.  Since $\Gamma \subseteq G \times H$ is irreducible, the projection $G \times H \to H$ is injective when restricted to $\Gamma$, so that we can conclude that $g^{-1}g' g$ and hence also $g'$ is trivial.
\end{proof}
Concrete examples of irreducible lattices in products of connected, simply connected, nilpotent Lie groups can be obtained from arithmetic constructions.  We provide a class of examples based on the Heisenberg group and mention the work of Machado \cite{machado2020-approximate-lattices-nilpotent}, which shows that all approximate lattices in connected, simply connected, nilpotent Lie groups are of arithmetic origin.
\begin{example}
  \label{ex:arithmetic-lattice-classifiable}
  Denote by $\rH_n$ the Heisenberg group of dimension $2n + 1$.  Let $d \in \ZZ$ be a square-free integer and consider the embedding
  \begin{gather*}
    \ZZ(\sqrt{d}) \to \RR \times \RR \colon
    a+b\sqrt{d} \mapsto (a+b\sqrt{d} , a-b\sqrt{d})
    \eqstop
  \end{gather*}
  It induces an embedding of the Heisenberg groups
  \begin{gather*}
    \rH_n(\ZZ(\sqrt{d})) \to \rH_n(\RR) \times \rH_n(\RR)
  \end{gather*}
  whose image is an irreducible lattice.  We identify $\rH_n(\ZZ(\sqrt{d}))$ with this lattice.  Taking for $W$ any metric ball in $\rH_n(\RR)$, we obtain the regular model set $\Lambda = \pi_1(\rH_n(\ZZ(\sqrt{d})) \cap (\rH_n(\RR) \times W))$, where $\pi_1$ is the first factor projection.  \Cref{cor:irreducible-lattice-classifiable} applies and gives rise to a countable family of classifiable \Cstar-algebras parametrised by the pairs $(n,d)$ for $n \geq 1$ and $d$ a square-free integer.  We note also that instead of $\ZZ[\sqrt{d}]$ one could consider the ring of algebraic integers in $\QQ(\sqrt{d})$, which is a finite index extension if $d \equiv 1 (\mathrm{mod} 4)$.

  More examples can be obtained by following the construction of \cite[Proof of Theorem 1.5, necessity]{machado2020-approximate-lattices-nilpotent} and considering Lie algebras over number fields.  Their classification up to dimension 6 is given in \cite{degraaf2007}.
\end{example}
As we obtain classifiable \Cstar-algebras, it is intriguing to calculate their Elliott invariant and thus determine their isomorphism class.  We are able to achieve the following partial results in this direction. We remark that, descriptions of $\rK$-theory for groupoids associated with cut-and-project schemes in Euclidean groups have been obtained in \cite[Theorem 4.2]{forresthuntonkellendonk2002}. See also \cite[Chapters III, IV and V]{forresthuntonkellendonk2002} for concrete calculations.
\begin{proposition}
  \label{prop:elliott-invariant-model-set}
  Let $\Lambda$ be a regular model set in a conneceted, simply connected nilpotent Lie group arising from an irreducible lattice as in \cref{cor:irreducible-lattice-classifiable}.  Then
  \begin{itemize}
  \item $\rK_*(\Cstar(\Lambda)) \cong \rK^*(\Omega(\Lambda))$ as abelian groups, and
  \item $\Cstar(\Lambda)$ has a unique trace.
  \end{itemize}
\end{proposition}
\begin{proof}
  Denote by $\Lambda \subseteq G$ the inclusion in the assumptions.  By \cref{lem:stable-isomorphism}, the \Cstar-algebras $\Cstar(\Lambda)$ and $\Cstarred(\cG(\Lambda))$ are stably isomorphic and hence have isomorphic $\rK$-theory.  As connected simply connected nilpotent Lie groups are amenable and as such satisfy the Baum-Connes conjecture \cite{higsonkasparov01} and $G$ is a model for its classifying space for proper actions, it follows that there is a $\mathrm{KK}^G$-equivalence $\cont(\Omega(\Lambda)) \sim_{\mathrm{KK}^G} \conto(G \times \Omega(\Lambda))$, where $G$ acts diagonally on $G \times \Omega(\Lambda)$.  Considering the self-homeomorphism of $G \times \Omega(\Lambda)$ defined by $(g,P) \mapsto (g, gP)$, we find that $G \grpaction{} G \times \Omega(\Lambda)$ is conjugate to the action in its first factor.  By descent in $\mathrm{KK}$-theory, we infer that
  \begin{gather*}
    \cont(\Omega(\Lambda)) \rtimes G
    \sim_{\mathrm{KK}}
    (\conto(G) \rtimes G ) \otimes \cont(\Omega(\Lambda))
    \cong
    \cK(\Ltwo(G)) \otimes \cont(\Omega(\Lambda))
    \eqstop
\end{gather*}
This shows that
\begin{gather*}
  \rK_*(\Cstar(\Lambda))
  \cong
  \rK_*(\cont(\Omega(\Lambda)) \rtimes G)
  \cong
  \rK_*(\cont(\Omega(\Lambda)))
  \cong
  \rK^*(\Omega(\Lambda))
  \eqstop
\end{gather*}

Let us now show that $\Cstar(\Lambda)$ has a unique trace.  Since $\cG(\Lambda)$ is principal, by \cite[Theorem 1.1]{neshveyev2013-kms-states-groupoids} (see also \cite{renault87-produits-croises}), traces on $\Cstar(\Lambda)$ are precisely of the form $\int \rmd \mu \circ \rE$ for $\cG(\Lambda)$ invariant probability measures $\mu \in \cP(\Omega_0(\Lambda))$, where $\rE: \Cstar(\Lambda) \to \cont(\Omega_0(\Lambda))$ is the natural conditional expectation.  So it suffices to prove that $\Omega_0(\Lambda)$ caries a unique $\cG(\Lambda)$-invariant probability measures.  Let $\mu$ be such measure and consider the finite measure $\mu_G$ on $\Omega(\Lambda)$ associated with it by transverse measure theory through the correspondence $\cG(\Lambda) \sim G \ltimes \Omega(\Lambda)$. If $\nu$ denotes an admissible probability measure on $G$, and $\tilde \mu_G$ is a probability measure equivalent to $\mu_G$, then $\frac{1}{n} \sum_{i = 1}^n \nu^{*n} * \tilde \mu_G$ is a $\nu$-stationary probability measure that is equivalent to $\mu_G$.  Consider the continuous map $\beta \colon \Omega(\Lambda) \to G \times H / \Gamma$ and the subset $\Omega(\Lambda)^{\mathrm{ns}}$ of non-singular points introduced in \cite{bjorklundhartnickpgorzelski2018}.  Then \cite[Theorem 3.4]{bjorklundhartnickpgorzelski2018} applied to $\frac{1}{n} \sum_{i = 1}^n \nu^{*n} * \tilde \mu_G$ shows that $\Omega(\Lambda)^{\mathrm{ns}}$ is $\mu_G$-negligible.  Hence $\beta_* \mu_G$ is a well-defined $G$-invariant $\sigma$-finite measure on $G \times H / \Gamma$.  So the argument of \cite[Lemma 3.7]{bjorklundhartnickpgorzelski2018} shows that $\mu_G$ is a scalar multiple of the unique $G \times H$-invariant probability measure on $G \times H/\Gamma$.  This proves uniqueness of $\mu$.
\end{proof}

\begin{remark}
  \label{rem:transverse-measure-theory}
  In \cite[Proposition 4.1]{enstadraum2022} we described how transverse measure theory associates finite measures on $\Omega_0(\Lambda)$ to $G$-invariant probability measures on $\Omega(\Lambda)$.  The proof of uniqueness of the trace on $\Cstar(\Lambda)$ in \cref{prop:elliott-invariant-model-set} needs the reverse construction though, which is why we presented a proof that the measure $\mu_G$ in there is finite.
\end{remark}

\begin{remark}
  \label{rem:ordered-k-theory}  
  In view of \cref{prop:elliott-invariant-model-set} and the continuous map $\beta\colon \Omega(\Lambda) \to G \times H / \Gamma$ from \cite{bjorklundhartnickpgorzelski2018}, which is bijective on the set of non-singular points, it would seem natural to expect that $\rK_\bullet(\Cstar(\Lambda)) \cong \rK^\bullet(G \times H / \Gamma)$.  While this might hold as abstract groups, such isomorphism cannot respect the order of $\rK$-theory.  Indeed, $\Omega_0(\Lambda) = \cG(\Lambda)\nought$ is totally disconnected by \cref{prop:FLC-implies-tdlc} and carries an invariant probability measure so that there are infinite strictly descending chains in the positive cone of $\rK_\bullet(\Cstar(\Lambda))$.  As $G \times H / \Gamma$ is a closed manifold, its $\rK$-theory is finitely generated.
\end{remark}

\begin{problem}
  \label{prob:elliot-invariant}
  Calculate the Elliott invariant of $\Cstar(\Lambda)$ for a regular model set $\Lambda$ in a conneceted, simply connected nilpotent Lie group arising from an irreducible lattice as in \cref{cor:irreducible-lattice-classifiable}.
\end{problem}




{\small
  \printbibliography
}


\vspace{2em}

\begin{center}
\begin{minipage}[t]{0.33\linewidth}
  \small
  Ulrik Enstad \\
  Department of Mathematics \\
  University of Oslo \\
   Moltke Moes vei 35 \\
  0851 Oslo \\
  Norway \\[0.5em]
  ubenstad@math.uio.no
\end{minipage}
\begin{minipage}[t]{0.33\linewidth}
  \small
  Gabriel Favre \\
  Department of Mathematics \\
  Stockholm University \\
  SE-106 91 Stockholm \\
  Sweden \\[0.5em]
  favre@math.su.se
\end{minipage}
\begin{minipage}[t]{0.33\linewidth}
  \small
  Sven Raum \\
  Institut für Mathematik \\
  Universit{\"a}t Potsdam \\
  Campus Golm, Haus 9 \\
  Karl-Liebknecht-Str. 24-25 \\
  D-14476 Potsdam OT Golm \\
  Germany \\[0.5em]
  sven.raum@uni-potsdam.de
\end{minipage}
\end{center}
\end{document}